\documentclass[12pt]{amsart}

\usepackage{amssymb} 
\usepackage{color}
\usepackage{paralist} 
\usepackage{caption}
\usepackage{subfigure}
\usepackage{thm-restate}

\usepackage{tikz}
\tikzstyle{punkt}=[circle, fill=black, minimum size=1mm,inner sep=0pt, draw]

\usepackage[colorlinks]{hyperref}
\usepackage{cleveref} 

\newtheorem{theorem}{Theorem}[section]
\newtheorem{proposition}[theorem]{Proposition}

\newtheorem{corollary}[theorem]{Corollary}
\newtheorem*{theorem*}{Theorem}

\theoremstyle{definition}
\newtheorem{definition}[theorem]{Definition}

\newtheorem{example}[theorem]{Example}
\newtheorem{problem}[theorem]{Problem}

\newtheorem{conjecture}[theorem]{Conjecture}
\newtheorem{remark}[theorem]{Remark}

\newcommand{\NN}{ \ensuremath{\mathbb{N}}}

\newcommand{\KK}{ \ensuremath{\mathbb{K}}}
\newcommand{\ac}{\mathfrak{a}}
\DeclareMathOperator{\reg}{reg}
\newcommand{\Hom}{\mathrm{Hom}}
\DeclareMathOperator{\Tor}{Tor}
\DeclareMathOperator{\pdim}{pdim}
\DeclareMathOperator{\Spec}{Spec}
\DeclareMathOperator{\init}{in}

\newcommand{\height}{\mathrm{height}}

\DeclareMathOperator{\depth}{depth}
\DeclareMathOperator{\lk}{lk}
\DeclareMathOperator{\Min}{Min}
\newcommand{\pc}{\mathfrak{p}}
\DeclareMathOperator{\lin}{lin}

\newcommand{\set}[1]{\{#1\}}
\newcommand{\with}{\colon}

\begin{document}

\title{Algebraic properties of ideals of poset homomorphisms}

\author{Martina Juhnke-Kubitzke}
\address{Martina Juhnke-Kubitzke, Universit\"at Osnabr\"uck, Institut f\"ur Mathematik, 49069 Osnabr\"uck, Germany}
\email{juhnke-kubitzke@uni-osnabrueck.de}
\author{Lukas Katth\"an}
\address{Lukas Katth\"an, Goethe-Universit\"at Frankfurt, Institut f\"ur Mathematik, 60054 Frankfurt am Main, Germany}
\email{katthaen@math.uni-frankfurt.de}
\author{Sara Saeedi Madani}
\address{Sara Saeedi Madani, Universit\"at Osnabr\"uck, Institut f\"ur Mathematik, 49069 Osnabr\"uck, Germany}
\email{sara.saeedimadani@uni-osnabrueck.de}


\thanks{The first author and the third author were supported by the German Research Council DFG-GRK~1916.}

\begin{abstract}
Given finite posets $P$ and $Q$, we consider a specific ideal $L(P,Q)$, whose minimal monomial generators correspond to order-preserving maps $\phi:P\rightarrow Q$. 
We study algebraic invariants of those ideals. In particular, sharp lower and upper bounds for the Castelnuovo-Mumford regularity and the projective dimension are provided. Precise formulas are obtained for a large subclass of these ideals. 
Moreover, we provide complete characterizations for several algebraic properties of $L(P,Q)$, including being Buchsbaum, Cohen-Macaulay, Gorenstein, Golod and having a linear resolution. 
\end{abstract}

\maketitle

\section{Introduction}
The study of monomial ideals and the interplay between their algebraic and combinatorial properties is a central topic in combinatorial commutative algebra and algebraic combinatorics. Often there exists a one-to-one correspondence between a class of squarefree monomial ideals and certain combinatorial objects. One of the most prominent and classical examples for such a correspondence is provided by simplicial complexes and their Stanley-Reisner ideals. Other instances for such a relationship comprise edge and path ideals of graphs, as well as face and facet ideals of simplicial complexes.\\
In this paper, we are interested in the algebraic properties of a class of monomial ideals that are associated to two finite posets $P$ and $Q$.
Those ideals were introduced in a recent paper \cite{FGH} by Fl{\o}ystad, Greve and Herzog and further studied in \cite{HQS}.
Special classes of these ideals are closely related to generalized Hibi ideals and multichain ideals, as studied in \cite{EHM}.

Given posets $P$ and $Q$, we denote by $S_{P,Q}$ the polynomial ring over a field $\KK$ with variables $x_{p,q}$, where $p\in P$ and $q\in Q$. To any order-preserving map $\phi:P\rightarrow Q$, one associates the monomial
\[ u_\phi := \prod_{p \in P} x_{p,\phi(p)}. \]
The ideal associated to $P$ and $Q$, as defined in \cite{FGH}, is the monomial ideal $L(P,Q)$ in $S_{P,Q}$ that is generated by all monomials $u_\phi$ corresponding to order-preserving maps $\phi: P \rightarrow Q$. We denote $S_{P,Q}/L(P,Q)$ by $\KK[P,Q]$. 
If $P$ and $Q$, respectively, is a chain, then $L(P,Q)$ is called a \emph{letterplace ideal} and \emph{co-letterplace ideal}, respectively, in \cite{FGH}.

It was shown in \cite{FGH} that many interesting and known classes of monomial ideals, e.g., generalized ideals of Hibi type \cite{EHM}, initial ideals of determinantal ideals, Ferrers ideals and uniform face ideals, arise from the ideals $L(P,Q)$ via reduction with particular regular sequences. Besides an intrinsic interest in the properties of the ideals $L(P,Q)$, this connection provides an additional motivation for their study, since a lot of algebraic invariants and properties, as Castelnuovo-Mumford regularity, projective dimension, Cohen-Macaulay-ness and Gorenstein-ness are preserved by taking regular quotients. 

Driven by this objective, we first investigate some basic properties of these ideals in \Cref{sect:Reduction}.
More precisely, we show that the computation of the Castelnuovo-Mumford regularity and the projective dimension of $L(P,Q)$ can be reduced to achieving the same task for $L(P_i,Q_j)$, where $P_i$ ($1\leq i\leq s$) and $Q_j$ ($1\leq j\leq r$) denote the connected components of $P$ and $Q$, respectively (see \Cref{cor:disconnectedreg}).
Moreover, in \Cref{subsect:AlexDuality}, we determine all minimal prime ideals of the ideal $L(P,Q)$ (\Cref{prop:minprimes} and \Cref{cor:minprimes}). This allows us to completely characterize those pairs of posets $P$ and $Q$ such that the ideal $L(P,Q)$ is unmixed (\Cref{prop:dualclass}). 
The unmixed case is of particular importance since, under the assumption that $P$ is connected, unmixedness of $L(P,Q)$ is equivalent to the Alexander dual of $L(P,Q)$ coinciding with $L(Q,P)$ (up to switching indices of the variables), see \cite[Corollary~1.2 (2)]{HQS}.
In \Cref{sect:SourceTarget}, we study the behavior of the graded and total Betti numbers of $L(P,Q)$, when we restrict the poset $Q$ and $P$, respectively, to a subposet (see \Cref{cor:restrTarget} and \Cref{lem:extension}).

In \Cref{sect:AlgInv}, we focus on the study of the Castelnuovo-Mumford regularity, the projective dimension and the length of the linear strand of $L(P,Q)$.
We first  provide sharp lower and upper bounds for the Castelnuovo-Mumford regularity and the projective dimension of $L(P,Q)$ for arbitrary $P$ and $Q$ (see \Cref{cor:easybound} and \Cref{cor:lowerboundreg}).
Under the additional assumption that $L(P,Q)$ is unmixed, we are able to generalize Corollary 3.3. from \cite{EHM} and to provide the following precise formulas for the Castelnuovo-Mumford regularity and the projective dimension of $L(P,Q)$. 
In the following \namecref{thm:regformula}, $\ac(Q)$ denotes the maximal cardinality of an antichain in $Q$.
\begin{restatable*}{theorem}{Regularity}\label{thm:regformula}
 Let $P$ and $Q$ be posets such that $L(P,Q)$ is unmixed. Let further $s$ and $r$ be the number of connected components of $P$ and $Q$, respectively. Then
 \begin{itemize}
  \item[{\em(a)}] $\reg \KK[P,Q]=\ac(Q)\cdot (|P|-s)+s-1$.
  \item[{\em(b)}] $\pdim \KK[P,Q]=\ac(P)\cdot (|Q|-r)+s(r-1)+1$.
 \end{itemize}
\end{restatable*}

In the last part of \Cref{sect:AlgInv}, we are interested in the length of the linear strand of $L(P,Q)$.
Depending on the existence of certain subposets of $Q$, we give sharp lower bounds for this length (\Cref{prop:linear}). 
Coming back to our original motivation that a lot of interesting classes of ideals can be obtained as regular quotients of the ideals $L(P,Q)$ for special posets $P$ and $Q$ (see \cite[Section 3]{FGH}), we apply the previously obtained results to initial ideals of determinantal ideals and to initial ideals of $2$-minors of a symmetric matrix. 

In \Cref{sect:charact}, our aim is to characterize when the ideal $L(P,Q)$ has a certain algebraic property. 
One of our main results in this direction (\Cref{cor:charlinres}) classifies those pairs of posets $P$ and $Q$ such that  the ideal $L(P,Q)$ has a linear resolution.
\begin{restatable*}{theorem}{Polymatroidal}\label{cor:charlinres}
	Let $P$ and $Q$ be posets. Then the following conditions are equivalent:
	\begin{enumerate}
		\item[{\em(a)}] ${L(P,Q)}^k$ is weakly polymatroidal, for each positive integer $k$.
		\item[{\em(b)}] $L(P,Q)$ has linear quotients.
		\item[{\em(c)}] $L(P,Q)$ has a linear resolution.
		\item[\emph{(d)}] $\beta_{1,|P|+i}(L(P,Q)) = 0$ for $i \neq 1$.
		\item[{\em(e)}] $P$ is an antichain or $Q$ is a chain.
	\end{enumerate}
\end{restatable*}
We refer the reader to \Cref{sect:charact} for unexplained terminology.
Similarly, \Cref{cor:CM} characterizes those posets such that $\KK[P,Q]$ is Cohen-Macaulay.
\begin{restatable*}{theorem}{CohenMacaulay}\label{cor:CM}
	Let $P$ and $Q$ be posets such that $|Q|\geq 2$. Then the following conditions are equivalent:
	\begin{enumerate}
		\item[{\em(a)}] $\Delta(P,Q)$ is (pure) shellable.
		\item[{\em(b)}] $\KK[P,Q]$ is Cohen-Macaulay.
		\item[{\em(c)}] $\KK[P,Q]$ satisfies Serre's condition $(S_r)$ for some $r\geq 2$.
		\item[{\em(d)}] $\KK[P,Q]$ is sequentially Cohen-Macaulay.
		\item[{\em(e)}] $P$ is connected, and, in addition,  $P$ is a chain or $Q$ is an antichain.
	\end{enumerate}
\end{restatable*}
In the above theorem, $\Delta(P,Q)$ is the simplicial complex whose Stanley-Reisner ideal is $L(P,Q)$.
\Cref{cor:charlinres} and \Cref{cor:CM} imply that for $L(P,Q)$ to have some desirable algebraic property, it is essentially necessary that either $P$ or $Q$ is a chain, see \Cref{rem:LPcase}.

Next, in \Cref{cor:Gor} and \Cref{cor:Buchsbaum} we provide similar results for $\KK[P,Q]$ being Gorenstein and Buchsbaum (but not Cohen-Macaulay), respectively.
Moreover, in \Cref{prop:golod} we provide a characterization of those pairs of posets $P$ and $Q$ such that $\KK[P,Q]$ has the Golod property. 
In \Cref{sect:questions}, we collect some open problems and questions.

\section{Notations and conventions}

\subsection{Partial orders}

In this section, we give a short introduction to posets and their combinatorics. For more information on this topic we refer the reader to \cite{StanleyEnumerative}. 
Throughout this paper, we set $[n] := \set{1, \dotsc, n}$ for $n \in \NN$ and consider this set with the total order $1 < 2 < \dotsb < n$. So as a poset, $[n]$ is a chain with $n$ elements. 
Let $P$ be a partially ordered set (\emph{poset}, for short) on the ground set $\{p_1,\ldots,p_m\}$. For elements $p\leq q$ in $P$, we denote by $[p,q]$ the closed interval $\{u\in P~:~ p\leq u \leq q\}$, endowed with the partial order induced by $P$. Similarly, we use $(p,q)$ to denote the open interval $\{u\in P~:~ p< u < q\}$. One says that $p$ \emph{covers} $q$ if there is no $u\in P$ with $u\neq p,q$ such that $p< u<q$. 
To a poset $P$ one associates its \emph{Hasse diagram}, which is the directed graph on vertex set $V(P)=\{p_1,\ldots, p_m\}$ and with the set of edges $E(P)=\{(p_i,p_j)~:~p_j\mbox{ covers }p_i\}$.  Usually, the directions of the edges in a  Hasse diagram are omitted and just indicated by placing elements above than the elements they cover, see \ref{fig} for examples. 
We call a poset $P$ \emph{connected} if its Hasse diagram is connected as a graph. The posets corresponding to a single Hasse diagram in \ref{fig} are connected, whereas the poset, whose Hasse diagram is the union of all three Hasse diagrams in \ref{fig} is disconnected. 
Any disconnected poset $P$ can be written as a disjoint union of inclusion-maximal connected subposets, and these (uniquely determined) subposets are called the \emph{connected components} of $P$. 
If for two elements $p,q\in P$ there exists a smallest upper bound, then we denote it by $p\vee q$ and call it the \emph{join} of $p$ and $q$. 
An \emph{antichain} of a poset $P$ is a subset $S$ of $P$ such that elements in $S$ are pairwise incomparable. We use $\ac(P)$ to denote the maximal cardinality of an antichain in $P$. 
A subset $C$ of a poset $P$ is called a \emph{chain}, if any two elements of $C$ are comparable in $P$. The \emph{length} of a (finite) chain $C$ is equal to $|C|-1$. 
The so-called \emph{order complex} $\Delta(P)$ of a poset $P$ is the abstract
simplicial complex on vertex set $P$, whose $i$-dimensional faces are the chains of $P$ of length $i$. 
Given simplicial complexes $\Delta$ and $\Gamma$, their \emph{join} is the simplicial complex $\Delta\ast \Gamma=\{F\cup G~:~F\in \Delta,\;G\in \Gamma\}$. 

We now provide some definitions that are central for this article. 
Let $P$ and $Q$ be two (finite) posets.
A map $\phi: P \to Q$ is called \emph{order-preserving} (or \emph{isotone}) if for all $a,b \in P$ with $a\geq b$ it holds that $\phi(a) \geq \phi(b)$.
We write $\Hom(P,Q)$ for the set of order-preserving maps from $P$ to $Q$.
To each $\phi \in \Hom(P,Q)$, we associate the monomial
\[ u_\phi := \prod_{p \in P} x_{p,\phi(p)} \]
in the polynomial ring $S_{P,Q} := \KK[x_{p,q} \colon p \in P, q \in Q]$, where $\KK$ is an arbitrary field.
Following \cite{FGH}, we further set
\[ L(P,Q) := \left( u_\phi \colon \phi \in \Hom(P,Q) \right) \subseteq S_{P,Q} \]
and $\KK[P,Q] := S_{P,Q} / L(P,Q)$.
We denote by $\Delta(P,Q)$ the simplicial complex on the vertex set $P \times Q$, whose Stanley-Reisner ring is $\KK[P,Q]$.

\subsection{Minimal free resolutions}\label{sect:Alg}

In the following, let $\KK$ be a fixed arbitrary field.
We use $S$ to denote an arbitrary polynomial ring over $\KK$.
All polynomial rings over $\KK$ are endowed with the fine grading, i.e., if $S=\KK[x_1, \dotsc, x_n]$, then the variable $x_{i}$ has degree $\mathbf{e}_{i} \in \NN^{n}$. Here, $\mathbf{e}_{i}$ denotes the unit vector, whose $i$\textsuperscript{th} entry equals $1$. 
Further, for any monomial ideal $I$ in a polynomial ring $S$, we write $G(I)$ for its minimal set of monomial generators.

Associated to $S/I$ is a unique (up to isomorphism)  \emph{minimal
free multigraded resolution} of $S/I$ of the form 
$$0\rightarrow\bigoplus_{\mathbf{a} \in \NN^n}S(-\mathbf{a})^{\beta_{p,\mathbf{a}}^S(S/I)}\rightarrow\cdots\rightarrow\bigoplus_{\mathbf{a} \in \NN^n}
S(-\mathbf{a})^{\beta_{1,\mathbf{a}}^S(S/I)}\rightarrow S\rightarrow0$$,
where $p\leq n$ and $S(-\mathbf{a})$ is the $S$-module obtained by
shifting the degrees of $S$ by $\mathbf{a} \in \NN^n$. The numbers $\beta_{i,\mathbf{a}}^S(S/I)$, which are called \emph{multigraded Betti numbers} of $S/I$ 
equal the number of minimal generators of multidegree $\mathbf{a}$ in the
$i$\textsuperscript{th} syzygy module. The \emph{projective dimension} of $S/I$, denoted $\pdim_S S/I$, is equal to $p$. 

Moreover, for $i \in \NN$ and $\mathbf{a} \in \NN^n$, the multigraded Betti numbers are given also by
\[ \beta_{i,\mathbf{a}}^S(S/I) = \dim_\KK \Tor_i^S(S/I,\KK)_\mathbf{a}.\]
 Similarly, for $i\in \NN$ and $j\in \NN$, the numbers 
\[\beta_{i,j}^S(S/I)=\sum_{\substack{\mathbf{a}\in \NN^n\\ \sum_{k}a_k=j}}\beta_{i,\mathbf{a}}^S(S/I)\]
are called the \emph{graded Betti numbers} and $\beta_i^S(S/I)=\sum_{j\in \NN}\beta_{i,j}^S(S/I)$ denotes the $i$\textsuperscript{th} total Betti number of $S/I$. 

The \emph{Castelnuovo-Mumford regularity} of $S/I$ is defined to be 
\[
\reg_S S/I=\max \{j-i:\beta_{i,j}^S(S/I)\neq 0\}.
\]

Suppose that all the elements of 
$G(I)$ have the same degree $d$. Then $I$ is said to have a \emph{linear resolution} if
for all $i\in\NN$, $\beta_{i,j}^S(S/I)=0$ for all $j\neq{i+d-1}$.

The \emph{multigraded Betti numbers} of $I$ are defined as $\beta_{i,\mathbf{a}}^S(I)=\beta_{i+1,\mathbf{a}}^S(S/I)$ for $i\in \NN$ and $\mathbf{a}\in \NN^n$. Consequently, one has 
$\pdim^S I=\pdim^S S/I-1$ and $\reg^S I=\reg^S S/I+1$.
 
If it is clear from the context in which ring we are working, we will omit $S$ from the notation of the Betti numbers, projective dimension and regularity, and just write $\beta_{i,j}(S/I)$, $\pdim S/I$ and $\reg S/I$.

\section{Auxiliary results}\label{sect:BasicProp}
In this section, our aim is to deduce some useful properties of the ideals $L(P,Q)$, which will be used in the proofs of the following sections. 

\subsection{Reduction to the connected components}\label{sect:Reduction}
\newcommand{\ab}{\mathbf{a}}
\newcommand{\bb}{\mathbf{b}}

In the following,  we consider posets $P$ and $Q$ with the connected components $P_1,\ldots,P_s$ and $Q_1,\ldots,Q_r$, respectively. The aim of this section is to show that the computation of the Castelnuovo-Mumford regularity and the projective dimension of  $\KK[P,Q]$ can be reduced to the ones of $\KK[P_i,Q_j]$, for $1\leq i\leq s$, $1\leq j\leq r$. For this we need the following general formulas for monomial ideals in disjoint sets of variables.
The second part of the next proposition already appeared in Corollary 2.2 and Corollary 2.3 of \cite{JK}. The statements concerning the Castelnuovo-Mumford regularity and the projective dimension in both parts are also shown in Lemma 3.2 and Lemma 2.2, respectively, of \cite{Hoa}.  
Nevertheless, as both parts are proven similarly, we include a complete proof of the proposition.

\begin{proposition}\label{prop:disjoint}
Let $S = \KK[x_1, \dotsc,x_n, y_1\dotsc, y_m]$ be a polynomial ring. Let $I$ and $J$ be two monomial ideals, such that the elements in $G(I)$ and $G(J)$ involve only the variables $x_i$ ($1\leq i\leq n$) and $y_i$ ($1\leq i\leq m$), respectively.
\begin{enumerate}
\item[{\em(a)}] For $i\geq 1$, $\ab \in \NN^n, \bb \in \NN^m$ the following holds:
	\begin{align*}
		\beta^S_{i, (\ab,\bb)} ( S / IJ) &= \sum_{\substack{j,k \geq 1 \\ (j-1)+(k-1)=i-1}} \beta^S_{j,\ab}(S/I) \cdot \beta^S_{k,\bb}(S/J) \\
		\pdim_S S/IJ &= \pdim_S S/I + \pdim_S S/J - 1\\
		\reg_S S/IJ &= \reg_S S/I + \reg_S S/J + 1.
	\end{align*}
\item[{\em(b)}] For $i \geq 0$ and $\ab \in \NN^n, \bb \in \NN^m$ the following holds:
	\begin{align*}
		\beta^S_{i, (\ab,\bb)} ( S /(I+J)) &= \sum_{\substack{j,k \geq 0 \\j+k=i}} \beta^S_{j,\ab}(S/I) \cdot \beta^S_{k,\bb}(S/J) \\
		\pdim_S S/(I+J) &= \pdim_S S/I + \pdim_S S/J\\
		\reg_S S/(I+J) &= \reg_S S/I + \reg_S S/J.
	\end{align*}
\end{enumerate}
\end{proposition}

\begin{proof}
For any monomial ideal $\mathfrak{b} \subseteq S$, we write $L_\mathfrak{b}$ for its lcm-lattice.
Recall that $\beta^S_{i, \ab}(S/\mathfrak{b}) = \dim_\KK \tilde{H}_{i-2}(\Delta((\hat{0},\ab)_{L_\mathfrak{b}}); \KK)$ for $\ab \in L_\mathfrak{b}\setminus \{\hat{0}\}$ and $i \geq 1$, (see \cite[Theorem~2.1]{GPW}).
\begin{asparaenum}[(a)]
\item If $\ab=0$ or $\bb=0$, both sides of the formula for the Betti numbers evaluate to $0$, so we are done in this case. Now, let $\ab\neq 0$ and $\bb\neq 0$.
It is easy to see that $L_{IJ} = (L_I \setminus \{\hat{0}\}) \times (L_J \setminus \{\hat{0}\}) \cup \{\hat{0}\}$, since the minimal generators of $I$ and $J$ are defined on disjoint sets of variables.
Further, for $\ab \in L_I, \bb \in L_J$, it follows from \cite[Theorem~5.1~(c)]{Walker} that $\left|\Delta((\hat{0},\ab \vee \bb)_{L_{IJ}})\right| \approx \left|\Delta((\hat{0},\ab)_{L_{I}})\right| * \left|\Delta((\hat{0},\bb)_{L_{J}})\right|$.
Here, $\left|\cdot\right|$ denotes the geometric realization of the order complex, ``$\approx$'' means homeomorphic and ``$*$'' denotes the join.
Now the claim follows from a straightforward computation:
\begin{align*}
\beta^S_{i, (\ab,\bb)} ( S / IJ) &= \dim_\KK \tilde{H}_{i-2}(\Delta((\hat{0},\ab \vee \bb)_{L_{IJ}}); \KK) \\
&= \dim_\KK \tilde{H}_{i-2}(\Delta((\hat{0},\ab)_{L_{I}}) * \Delta((\hat{0},\bb)_{L_{J}});\KK) \\
&= \dim_\KK \bigoplus_{(j-1)+(k-1)=i-1} \tilde{H}_{j-2}(\Delta((\hat{0},\ab)_{L_{I}});\KK) \otimes_{\KK} \tilde{H}_{k-2}(\Delta((\hat{0},\bb)_{L_{J}});\KK) \\
&= \sum_{(j-1)+(k-1)=i-1} \beta^S_{j,\ab}(S/I) \cdot \beta^S_{k,\bb}(S/J).
\end{align*}
In the third line, we use the algebraic K\"unneth formula, cf. \cite[Theorem 3B.5]{Hat}. For this, note that the simplicial chain complex of a join is just the tensor product of the chain complexes of the factors. 
The formulas for the projective dimension and the regularity are an immediate consequence of the formula for the Betti numbers. 
\item Note that $L_{I+J} = L_I \times L_J$. 
By \cite[Theorem 5.1 (d)]{Walker} it holds that 
\[ |\Delta((\hat{0},\ab \vee \bb)_{L_{I+J}})| \approx \begin{cases}
	\mathrm{susp}(|\Delta((\hat{0},\ab)_{L_{I}})| * |\Delta((\hat{0},\bb)_{L_{J}})|) &\text{ if } \ab \neq 0, \bb \neq 0; \\
	|\Delta((\hat{0},\ab)_{L_{I}})| &\text{ if } \bb = 0; \\
	|\Delta((\hat{0},\bb)_{L_{J}})| &\text{ if } \ab = 0,
\end{cases} \]
where $\mathrm{susp}(\cdot)$ denotes the suspension.
The formula for the Betti numbers is trivially true if $i = 0$, so we may assume that $i \geq 1$.
Further, note that if $\ab \neq 0$ and $\bb \neq 0$, then in the right-hand side of the formula for the Betti numbers, only the terms with $j,k\geq 1$ contribute.
In this case, the claim follows analogously to the first part.
On the other hand, if $\bb = 0$, then
\begin{align*}
\beta^S_{i, (\ab,\bb)} ( S / (I+J)) &= \dim_\KK \tilde{H}_{i-2}(\Delta((\hat{0},\ab \vee \bb)_{L_{I+J}}); \KK) = \dim_\KK \tilde{H}_{i-2}(\Delta((\hat{0},\ab)_{L_{I}});\KK) \\
	&= \beta^S_{i, \ab} ( S / I) = \beta^S_{i, \ab} ( S / I)\beta^S_{0, \bb} ( S / J) \\
	&= \sum_{\substack{j,k \geq 0 \\j+k=i}} \beta^S_{j,\ab}(S/I) \cdot \beta^S_{k,\bb}(S/J)
\end{align*}
and similar for $\ab = 0$.
\end{asparaenum}
\end{proof}

We can now state the mentioned reduction to connected components.  In the following, given an ideal $I$ we write $\Min(I)$ for the set of its minimal prime ideals. 

\begin{corollary}\label{cor:disconnectedreg}
Let $P$ and $Q$ be posets with  the connected components $P_1,\ldots,P_s$ and $Q_1,\ldots,Q_r$, respectively. Then
\[ L(P,Q) = \prod_{i=1}^s L(P_i, Q) = \prod_{i=1}^s \sum_{j=1}^r L(P_i, Q_j). \]
In particular, the following holds:
\begin{itemize}
\item[{\em(a)}] $\displaystyle{\pdim \KK[P,Q] = \sum_{\substack{1\leq i\leq s\\ 1\leq j\leq r}} \pdim \KK[P_i, Q_j] - (s-1)}.$
\item[{\em(b)}] $\displaystyle{\reg \KK[P,Q] = \sum_{\substack{1\leq i\leq s\\ 1\leq j\leq r}} \reg \KK[P_i, Q_j] + (s-1)}.$
\item[{\em(c)}] $\Min(L(P,Q)) =\bigcup_{i=1}^s\Min(L(P_i,Q))$.
\end{itemize}
\end{corollary}
\begin{proof}
The first equality follows from the fact that there is a bijection between
$\Hom(P_1 \cup \dotsb \cup P_s, Q)$ and $\Hom(P_1,Q) \times \dotsb \times \Hom(P_s,Q)$.
Further, as $P_i$ is connected for each $i$ it holds that
\[ \Hom(P_i, Q_1 \cup \dotsb \cup Q_r) = \Hom(P_i,Q_1) \cup \dotsb \cup \Hom(P_i,Q_r),\]
which implies the second equality.
Now the formulas for the projective dimension and the regularity are immediate consequences of \Cref{prop:disjoint}.
Here we use that the regularity and the projective dimension do not change if we consider the ideals $L(P_i,Q_j) \subseteq S_{P_i,Q_j}$ in the larger ring $S_{P,Q}$.

For the last formula, note that the ideals $L(P_i,Q)$ are on disjoint sets of variables. Hence it holds that
\[ L(P,Q) = \prod_{i=1}^s L(P_i, Q) = \bigcap_{i=1}^s L(P_i, Q). \]
Moreover, $L(P,Q)$ and all $L(P_i,Q)$ are radical ideals and thus equal to the intersection of their minimal primes. So the last formula is clear.
\end{proof}

\begin{example}\label{ex:antichain}
As an application of \Cref{cor:disconnectedreg}, we consider the situation that one of the posets $P$ and $Q$ is an antichain.
\begin{asparaenum}[(a)]
\item Assume that $P=\{p_1,\ldots,p_s\}$ is an antichain and $Q$ is an arbitrary poset. Then $L(P,Q) = \prod^s_{i=1} L(\{p_i\}, Q)$.
Since $L(\{p_i\}, Q)$ is generated by the variables $x_{p_i,q}$ for $q \in Q$, it follows that $\pdim \KK[\{p_i\}, Q] = | Q|$ and $\reg \KK[\{p_i\}, Q] = 0$.
Hence, by \Cref{cor:disconnectedreg} it holds that
\begin{align*}
\pdim \KK[P,Q] &= |P| (|Q| - 1) +  1\mbox{ and} \\
\reg \KK[P,Q] &= |P| - 1.
\end{align*}
Note that as $L(P,Q)$ is generated in degree $|P|$, the regularity attains the minimal possible value, which forces $L(P,Q)$ to have a linear resolution.
\item Consider the case that $P$ is connected and $Q=\{q_1,\ldots,q_r\}$ is an antichain. Then $L(P,Q) = \sum_{j=1}^r L(P, \{q_j\})$.
Since $L(P, \{q_j\})$ is a principal ideal, whose generator is of degree $|P|$, it follows that $\pdim \KK[P,\{q_j\}] = 1$ and $\reg \KK[P,\{q_j\}] = |P |-1$. Then \Cref{cor:disconnectedreg} implies that
\begin{align*}
\pdim \KK[P,Q] &= |Q| \mbox{ and} \\
\reg \KK[P,Q] &= |Q|(|P| - 1).
\end{align*}
In fact, in this case $\KK[P,Q]$ is a complete intersection (see paragraph before \Cref{cor:Gor} for the definition of complete intersection). 
\item Finally, consider the case that $P$ is arbitrary with connected components $P_1,\ldots,P_s$ and $Q$ is an antichain.
We know from (b) that
$\pdim \KK[P_i,Q] = |Q|$ and $\reg \KK[P_i,Q] = |Q|(|P_i| - 1)$ for $1\leq i\leq s$.
Hence, using \Cref{cor:disconnectedreg}, we conclude that
\begin{align*}
\pdim \KK[P,Q] &= s\cdot (|Q|-1)+1 \mbox{ and} \\
\reg \KK[P,Q] &= |Q|(|P| - s)+s-1.
\end{align*}
\end{asparaenum}
\end{example}

\subsection{Minimal primes, Alexander duality and the unmixed case}\label{subsect:AlexDuality}
%
%
\newcommand{\pp}{\mathfrak{p}}
\newcommand{\MHom}{\operatorname{MHom}}
For certain classes of posets $P$ and $Q$, the Alexander dual of $L(P,Q)$ satisfies the relation $L(P,Q)^\vee = L(Q,P)^\tau$.
Here, the superscript $\tau$ indicates that $L(Q,P)^\tau$ is obtained from $L(Q,P)$ by switching the indices of the variables.
Indeed, in \cite{EHM} it was shown that this holds if one of the posets is a chain and, recently, in \cite{HQS} a complete characterization of the pairs $(P,Q)$ of posets satisfying $L(P,Q)^{\vee}=L(Q,P)^\tau$ was provided.
In these cases, it is shown that the minimal primes of $L(P,Q)$ are in one-to-one correspondence with order-preserving maps $Q \to P$. In the following, we aim at a description of the minimal prime ideals of $L(P,Q)$ for arbitrary poset $P$ and $Q$. For this, one needs to consider multi-valued order-preserving maps from $Q$ to $P$ that we discuss now. 
Since every monomial prime ideal $\pp \subseteq S_{P,Q}$ is generated by a set of variables, we can identify $\pp$ with the set
\[\psi_\pp := \set{(q,p) \in Q \times P \with x_{p,q} \in \pp} \subseteq Q \times P.\]
We consider $\psi_\pp$ as a multi-valued map from $Q$ to $P$. For $\psi \subseteq Q \times P$ and $q\in Q$ we set $\psi(q) := \set{p\in P\with (q,p) \in \psi}$ and 
$ \pp_\psi := (x_{p,q} \with (q,p) \in \psi)$.
\begin{definition}
	Let $\MHom(Q,P)$ be the collection of sets $\psi \subseteq Q \times P$, which satisfy the following two conditions:
	\begin{enumerate}[(a)]
		\item $\psi(q) \neq \emptyset$ for all $q \in Q$, i.e., every $q \in Q$ has at least one image under $\psi$, and
		\item for all $q,q' \in Q$ with $q < q'$, there exist $p \in \psi(q)$ and $p' \in \psi(q')$ with $p < p'$. 
	\end{enumerate}
	We call elements of $\MHom(Q,P)$ \emph{order-preserving multivalued maps} from $Q$ to $P$.
\end{definition}
Note there is a natural inclusion $\Hom(Q,P) \hookrightarrow \MHom(Q,P)$. 
\begin{proposition}\label{prop:minprimes}
	Let $P$ and $Q$ be posets. Then the following hold:
	\begin{enumerate}
		\item[{\em(a)}] For every $\pp \in \Min(L(P,Q))$, there exists a multivalued map $\psi \in \MHom(Q,P)$ such that $\pp = \pp_\psi$.
		\item[{\em(b)}] Let $\psi \in \MHom(Q,P)$.
		Then $L(P,Q) \subseteq \pp_\psi$ if and only if for every $\phi \in \Hom(P,Q)$, there exists an element $p \in P$ such that $p \in \psi(\phi(p))$.
	\end{enumerate}
\end{proposition}
\begin{proof}
This proof is mostly an adaption of the proof of \cite[Proposition~1.5]{FGH}.
We first prove part (b).
\begin{asparaenum}
\item[(b)] Let $\psi \in \MHom(Q,P)$ and $\phi \in \Hom(P,Q)$.
	Then $u_\phi \in \pp_\psi$ if and only if there exists $p \in P$ and $q \in Q$, such that $x_{p,q} \mid u_\phi$ and $x_{p,q} \in \pp_\psi$.
	In other words, $q = \phi(p)$ and $p \in \psi(q)$.
	Hence, $u_\phi \in \pp_\psi$ if and only if there exists a $p \in P$ such that $p \in \psi(\phi(p))$.

\item[(a)] Let $\pp \in \Min(L(P,Q))$. 
	We consider $\psi := \psi_{\pp}=\set{(q,p) \with x_{p,q} \in \pp}$. Since, by construction, $\pp_{\psi}=\pp$, it suffices to prove that $\psi \in \MHom(Q,P)$.
	
	First, we show that $\psi(q) \neq \emptyset$ for all $q \in Q$.
	For $q \in Q$ let $\phi_q: P \to Q$ be the constant map with image $q$.
	As $u_{\phi_q} \in L(P,Q) \subseteq \pp$, we deduce from part (b) that there exists an element $p \in P$ with $p \in \psi(\phi(p)) = \psi(q)$.
	
	It remains to verify that $\psi$ is order-preserving.
	Assume to the contrary that there exist $q, q' \in Q$ such that $q < q'$ and $p \nleq p'$ for all $p \in \psi(q)$, $p' \in \psi(q')$.
	Consider the map $\phi: P \to Q$ defined by
	\[ \phi(p) := 
	\begin{cases}
	q &\text{ if } \exists p'\in\psi(q')\text{ with } p \leq p',\\
	q' &\text{ else.}
	\end{cases}
	\]
	This map is clearly order-preserving. Now, if $\phi(p) = q$, then $p \notin \psi(q) = \psi(\phi(p))$ by construction.
	On the other hand, if $\phi(p) = q'$, then $p \notin \psi(q') = \psi(\phi(p))$.
	This yields a contradiction to part (b).
\end{asparaenum}
\end{proof}
\noindent We list some immediate consequences of the preceding proposition.
The following formula for the height generalizes part (a) of Corollary 1.2 in \cite{HQS}.
\begin{corollary}\label{cor:minprimes}
	Let $P$ and $Q$ be posets. Then the following hold:
	\begin{enumerate}[(a)]
		\item[{\em(a)}] Every $\pp \in \Min(L(P,Q))$ has height at least $|Q|$, and there exists a prime $\pp \in \Min(L(P,Q))$ with $\height(\pp) = |Q|$.
		\item[{\em(b)}] Let $\psi \in \Hom(Q,P)$. Then $\pp_\psi \in \Min(L(P,Q))$ if and only if for every $\phi \in \Hom(P,Q)$, the map $\psi \circ \phi$ has a fixed point.
	\end{enumerate}
	In particular, $\height (L(P,Q))=|Q|$ and $\dim \KK[P,Q]=|Q|\cdot(|P|-1)$.
\end{corollary}
\begin{proof}
\begin{asparaenum}[(a)]
\item Let  $\pp \in \Min(L(P,Q))$. By \Cref{prop:minprimes} (a), there exists a multivalued $\psi \in \MHom(Q,P)$ such that $\pp = \pp_\psi$. As $\psi(q) \neq \emptyset$ for all $q \in Q$, there are at least $|Q|$ variables contained in $\pp$.
Hence the height of $\pp$ is at least $|Q|$.

On the other hand, any constant map $\psi \in \Hom(Q,P)$ satisfies the condition of part (b) of \Cref{prop:minprimes}, so $L(P,Q) \subseteq \pp_\psi$. Since $\pp_\psi$ has height $|Q|$, it needs to be a minimal prime ideal and the claim follows. 
\item This is clear from part (b) of \Cref{prop:minprimes} and part (a).
\end{asparaenum}

The formulas for the height and the dimension are clear from part (a) and the formula $\dim(\KK[P,Q])=(|P|\cdot |Q|)-\height (L(P,Q))$.
\end{proof}
%
%
%
Next, we recall the characterization of those posets $P$ and $Q$, which satisfy $L(P,Q)^\vee = L(Q,P)^\tau$, as provided in \cite{HQS}.
\begin{figure}[h!]
\centering
\subfigure[]{\begin{tikzpicture}[scale = 0.6,every node/.style=punkt]
	\draw (-0.2,1) node {} -- (0.5,0) node {} -- (0.5,0) node {} -- (0.5,1) node {};
    \draw (0.5,0) node {} -- (1.2,1) node {};
\end{tikzpicture}\label{fig:rooted1}}\qquad\qquad
\subfigure[]{\begin{tikzpicture}[scale = 0.6,every node/.style=punkt]
\draw (0,0) node {} -- (-0.7,-1) node {};
\draw (0,0) node {} -- (0.7,-1) node {} -- (0,-2) node {};
\draw (0.7,-1) node {} -- (1.4,-2) node {};
\end{tikzpicture}\label{fig:rooted2}}\qquad\qquad
\subfigure[]{ \begin{tikzpicture}[scale = 0.6,every node/.style=punkt]
\draw (0,0) node {} -- (0,1) node {} -- (1,0) node {} -- (1,1) node {};
\end{tikzpicture}\label{fig:rooted3}}
\caption{}
\label{fig}
\end{figure}
Crucial for this characterization is the notion of a (co-)rooted poset 
, which we now recall. 
A poset $P$ is called \emph{rooted} if for any two incomparable elements $p_1, p_2\in P$, there is no element $p\in P$ such that $p > p_1, p_2$.
Similarly, $P$ is called \emph{co-rooted} if for any two incomparable elements $p_1, p_2 \in P$, there is no element $p \in P$ such that $p < p_1, p_2$.
It is easy to see that a connected poset $P$ is rooted (resp. co-rooted) if and only if its Hasse diagram (resp. the Hasse diagram of the dual of $P$) is a rooted tree as a directed graph (where edges are oriented from the smaller to the bigger element along a cover relation).
Moreover, note that an arbitrary poset is rooted if and only if each of its connected components is rooted. For example, consider the Hasse diagrams shown in \Cref{fig}.
The poset in \Cref{fig:rooted1} is rooted, but not co-rooted, whereas the poset in \Cref{fig:rooted2} is co-rooted, but not  rooted. The poset in \Cref{fig:rooted3} and the union of the posets in \Cref{fig:rooted1} and \Cref{fig:rooted2} are neither rooted nor co-rooted.
We are now in the position to state the mentioned characterization from \cite{HQS}.
\begin{proposition}[\cite{HQS}, Corollary 1.5]\label{prop:dualclass}
	Let $P$ and $Q$ be posets.
	Then $L(P,Q)^\vee=L(Q,P)^\tau$ if and only if $P$ or $Q$ is connected and one of the following conditions holds:
	\begin{itemize}
		\item[{\em(a)}] Both, $P$ and $Q$, are rooted.
		\item[{\em(b)}] Both, $P$ and $Q$, are co-rooted.
		\item[{\em(c)}] $P$ is connected and $Q$ is a disjoint union of chains.
		\item[{\em(d)}] $Q$ is connected and $P$ is a disjoint union of chains.
		\item[{\em(e)}] $P$ or $Q$ is a chain.
	\end{itemize}
\end{proposition}

Recall that an ideal $I$ in a polynomial ring $S$ is called \emph{unmixed} if all associated prime ideals of $I$ have the same height (see \cite[A.5]{HH}).
Similarly, $S/I$ is unmixed if $I$ is unmixed.
Note that if $P$ and $Q$ satisfy $L(P,Q)^\vee=L(Q,P)^\tau$, then, as $L(P,Q)^\vee$ is generated in a single degree, it follows that $L(P,Q)$ is unmixed. 
It hence should not be surprising that there is a close relation between the posets $P$ and $Q$ satisfying $L(P,Q)^\vee=L(Q,P)^\tau$ and those, for which $L(P,Q)$ is unmixed. 
The next result provides a complete characterization of pairs of posets falling into this latter class.

\begin{proposition}\label{pro:unmixed}
 Let $P$, $Q$ be posets. Then $L(P,Q)$ is unmixed if and only if one of the following conditions holds:
 \begin{itemize}
  \item[{\em(a)}] Both, $P$ and $Q$, are rooted.
  \item[{\em(b)}] Both, $P$ and $Q$, are co-rooted.
  \item[{\em(c)}] $P$ or $Q$ is a disjoint union of chains.
 \end{itemize}
Moreover, if $P$ or $Q$ is connected and $L(P,Q)$ is unmixed, then $L(P,Q)^\vee=L(Q,P)^\tau$.
\end{proposition}

\begin{proof}
%
%
Let $P_1,\ldots,P_s$ denote the connected components of $P$.
Then, by \Cref{cor:disconnectedreg}, we have $\Min(L(P,Q))=\bigcup_{i=1}^s\Min(L(P_i,Q))$. 
Therefore, $L(P,Q)$ is unmixed if and only if $L(P_i,Q)$ is unmixed for $i=1,\ldots,s$.
If $P_i$ is a chain for all $i$, then $L(P_i,Q)^\vee = L(Q, P_i)^\tau$ by \cite[Proposition~1.2]{FGH}.
Hence, $L(P, Q)$ is unmixed by the previous observation. 
Now consider the case that $P_i$ is not a chain for some $i$.
Using that by \Cref{cor:minprimes} $L(P,Q)$ is unmixed if and only if each element of $\Min(L(P,Q))$ has height $|Q|$, we deduce from the proof of \cite[Theorem 1.4]{HQS} the following statements:
\begin{itemize}
	\item If $P_i$ is rooted, then $L(P_i,Q)$ is unmixed if and only if $Q$ is rooted.
	\item If $P_i$ is co-rooted, then $L(P_i,Q)$ is unmixed if and only if $Q$ is co-rooted.
	\item If $P_i$ is neither rooted nor co-rooted, then $L(P_i,Q)$ is unmixed if and only if $Q$ is a disjoint union of chains. (Though the proof of \cite[Theorem 1.4]{HQS} only shows that if $Q$ is a disjoint union of chains, then $L(P_i,Q)^\vee = L(Q, P_i)^\tau$, this already implies that $L(P_i, Q)$ is unmixed, by the discussion preceding this proposition.) 
\end{itemize}
Note that the only posets which are both rooted and co-rooted are disjoint unions of chains.
In conclusion, it follows that $L(P,Q)$ is unmixed if and only if one of the following cases applies:
\begin{enumerate}
	\item $P_i$ is a chain for all $i$, or
	\item $P_i$ is rooted for all $i$ and $Q$ is rooted, or
	\item $P_i$ is co-rooted for all $i$ and $Q$ is co-rooted, or
	\item $Q$ is a disjoint union of chains.
\end{enumerate}
The first case is clearly equivalent to $P$ being a disjoint union of chains.
Further, since a poset is rooted (resp. co-rooted) if all its connected components are rooted (resp. co-rooted), the claimed classification of the unmixed case follows.

The ``Moreover''-statement follows from the just proven part of the proposition combined with \Cref{prop:dualclass}.
%
%
\end{proof}

\subsection{Restricting the target and the source}\label{sect:SourceTarget}
In this section, given two posets $P$ and $Q$, we consider subposets $P'\subseteq P$ and $Q'\subseteq Q$.
It is now natural to ask how algebraic invariants of $\KK[P',Q']$ are related to the corresponding invariants of $\KK[P,Q]$.
In the following, we are going to discuss this question.

\begin{proposition}\label{cor:restrTarget}
	Let $Q$ be a poset and $Q' \subseteq Q$ be a subposet.
	For all $i$ and all multidegrees $\ab \in \NN^{P \times Q'} \subseteq \NN^{P \times Q}$ the following holds:
	\[ \beta_{i,\ab}^{S_{P,Q}}(\KK[P,Q]) = \beta_{i,\ab}^{S_{P,Q'}}(\KK[P,Q']). \]
	In particular,
	\[ \beta_{i,j}^{S_{P,Q}}(\KK[P,Q]) \geq \beta_{i,j}^{S_{P,Q'}}(\KK[P,Q']) \]
	for all $i$ and $j$.\\
	Moreover, $\reg \KK[P,Q]\geq \reg\KK[P,Q']$ and $\pdim \KK[P,Q]\geq \pdim \KK[P,Q']$.
\end{proposition}
\begin{proof}
	It is enough to show the claim for $Q' = Q \setminus \{q_0\}$ for some element $q_0 \in Q$.
	In this case, the maps $P \to Q'$ are in one-to-one correspondence with the maps $P \to Q$, which do not hit $q_0$.
	Hence it holds that
	\[ \KK[P,Q'] \cong S_{P,Q} / \left(L(P,Q) + (x_{p,q_0} \colon p \in P)\right). \]
	Now the claim follows from Hochster's formula. Indeed, $\Delta(P,Q')$ is the induced subcomplex of $\Delta(P,Q)$ on vertex set $\{x_{p,q}~:~p\in P, q\in Q'\}$ and hence any induced subcomplex of $\Delta(P,Q')$ is also an induced subcomplex of $\Delta(P,Q)$.
\end{proof}

\begin{proposition}\label{lem:extension}
	Let $P$ be a poset and $P' \subseteq P$ be a subposet, such that every homomorphism $\phi: P' \to Q$ extends to a homomorphism $P \to Q$.
	Then
	\[ \beta_{i}^{S_{P,Q}}(\KK[P,Q]) \geq \beta_{i}^{S_{P',Q}}(\KK[P',Q]) \]
	for all $i$. 
\end{proposition}
\begin{proof}
	Consider the ring $R := \KK[P,Q][x_{p,q}^{-1} ~:~ p \in P \setminus P']$.
	As $R$ is a multigraded localization of $\KK[P,Q]$, it follows from \cite[Proposition~2.5]{Ei} that $\beta^{S_{P,Q}}_{i}(\KK[P,Q]) \geq \beta^{S_{P',Q}}_{i}(R)$ for every $i$. 
	On the other hand, it is easy to see that
	\[ R = S_{P,Q}[x_{p,q}^{-1} ~:~ p \in P \setminus P'] / I, \]
	where $I$ is the ideal generated by the monomials $u_\phi$ for those maps $\phi\in \Hom(P',Q)$ that are restrictions from maps $P \to Q$.
	From our assumption we infer that $I = L(P',Q)$ and hence $R$ is nothing but a Laurent polynomial extension of $\KK[P',Q]$.
	This does not affect the Betti numbers, i.e., we have $\beta^{S_{P',Q}}_{i,\ab}(R)=\beta^{S_{P',Q}}_{i,\ab}(\KK[P',Q]$) and the claim follows from the first part of the proof.
\end{proof}

\begin{remark}
\begin{asparaenum}[(a)]
\item We want ro remark that there exist subposets $P' \subseteq P$, which do not satisfy the condition of the above proposition.
Let $P$ and $Q$ be as in \Cref{figRestr}. 
%
\begin{figure}[h!]
\begin{tikzpicture}[scale = 0.6] 
	\begin{scope}
		\node at (0.5,-2) {$P$};
		\draw (0,0) node[punkt] {} -- ++(-60:1) node[punkt] {} -- ++(60:1) node[punkt] {};
	\end{scope}
	\begin{scope}[shift={(5,0)}]
		\node at (0.5,-2) {$Q$};
		\draw (0,-0.5) node[punkt] {}  ++(1,0) node[punkt] {};
	\end{scope}
\end{tikzpicture}
\caption{}
\label{figRestr}
\end{figure}
Let $P' \subseteq P$ be the antichain containing the two maximal elements of $P$.
Then the identity map $P'\to Q$ cannot be extended to a homomorphism $P \to Q$.
Further, one can compute that $\beta_{1}^{S_{P',Q}}(\KK[P',Q]) = 4$ while $\beta_{1}^{S_{P,Q}}(\KK[P,Q]) =2$,
so the conclusion of the preceding proposition does not hold in this case.
\item Also, note that \Cref{lem:extension} does not yield an inequality for the graded Betti numbers.
In general, $L(P,Q)$ is generated in degree $|P|$ whereas $L(P',Q)$ is generated in degree $|P'| < |P|$. Hence
\[ 0 = \beta_{1,|P'|}^{S_{P,Q}}(\KK[P,Q]) \ngeq \beta_{1,|P'|}^{S_{P',Q}}(\KK[P',Q]) > 0.\]
\end{asparaenum}
\end{remark}

\section{Algebraic invariants}\label{sect:AlgInv}
\subsection{Bounds for the regularity and the projective dimension}
In this section, we are interested in algebraic invariants of $\KK[P,Q]$ for general posets $P$ and $Q$.
%
%
%
Our first result are rough upper and lower bounds
for the projective dimension and the Castelnuovo-Mumford regularity of $\KK[P,Q]$.
\begin{proposition}\label{cor:easybound}
Let $P$ and $Q$ be posets. Then
\begin{itemize}
	\item[{\em(a)}] $|Q| \leq \pdim \KK[P,Q] \leq |P|(|Q|-1)+1$, and
	\item[{\em(b)}] $|P|-1 \leq \reg \KK[P,Q] \leq |Q|(|P|-1)$.
\end{itemize}
\end{proposition}
\begin{proof}
\begin{asparaenum}[(a)]
\item The lower bound for the projective dimension is immediate from \Cref{cor:minprimes}, because $\pdim\KK[P,Q]\geq \height(L(P,Q))$.
The upper bound for the projective dimension follows from  Hochster's formula, combined with the fact that $L(P,Q)$ is generated in degree $|P|$.

\item Since the ideal $L(P,Q)$ is generated in degree $|P|$, the lower bound for the regularity is obvious. For the upper bound, observe that by  \Cref{cor:minprimes} (a), all minimal primes of $L(P,Q)$ have height at least $|Q|$, which means that the minimal generators of the Alexander dual $L(P,Q)^\vee$ are of degree at least $|Q|$.  
Thus, by the same argument as above together with \cite[Proposition~8.1.8]{HH}, we have that $\reg \KK[P,Q] = \pdim L(P,Q)^\vee \leq |Q|(|P|-1)$.
\end{asparaenum}
\end{proof}

We want to remark that by \Cref{ex:antichain} all four bounds of \Cref{cor:easybound} are sharp.
Nevertheless, taking into account the number of connected components of $P$, the lower bound for the regularity can be improved significantly.
Recall that $\ac(Q)$ denotes the maximal cardinality of an antichain in $Q$.
\begin{proposition}\label{cor:lowerboundreg}
	Let $s$ be the number of connected components of $P$. Then
	\[ \reg \KK[P,Q] \geq \ac(Q)\cdot (|P|-s)+s-1. \]
	If, in addition, $s=1$ and $Q$ contains $a_k$ antichains of cardinality $k$, then
	\[\beta_{k, k|P|}^{S_{P,Q}}(\KK[P,Q]) \geq a_k.\]
\end{proposition}
\begin{proof} 
By \Cref{cor:disconnectedreg}, it suffices to consider the case that $P$ is connected, i.e., we have $s=1$.  
	We only prove the ``in addition''-part since the formula for the regularity directly follows from this by definition of regularity (see \Cref{sect:Alg}).  
	Let $A \subseteq Q$ be an antichain of cardinality $k$. 
	As $P$ is connected, it follows from \Cref{ex:antichain} (b) that $\KK[P,A]$ is a complete intersection.
	In particular, the Taylor resolution is indeed the graded minimal free resolution of $\KK[P,A]$ and hence $\beta_{k, P \times A}^{S_{P,A}} (\KK[P,A]) = 1$.
	
	Applying \Cref{cor:restrTarget} to all antichains of cardinality $k$ in $Q$ one obtains $a_k$ different multidegrees $m$ such that $\beta_{k,m}^{S_{P,Q}}(\KK[P,Q]) = 1$. Summing up these multigraded Betti numbers yields the result.
\end{proof}

In the special case that $L(P,Q)$ is unmixed, we can compute the regularity and the projective dimension explicitly.
\Regularity
%
\begin{proof}
(a) We prove the formula by induction on $|P|+|Q|$. If $|P|=1$ or $|Q|=1$, then it follows from \Cref{ex:antichain} (a) that $\reg \KK[P,Q]=0$, in the first case, and it follows from \Cref{ex:antichain} (c) that $\reg \KK[P,Q]=|P|-1$, in the second case. In particular, the formula for the regularity holds in these cases. Now let $|P|\geq 2$ and $|Q|\geq 2$.

We can further assume that both, $P$ and $Q$, are connected.  Indeed, once we have shown that the formula holds for all pairs of connected posets, it holds in general case as well. To see this, suppose that $P_1,\ldots,P_s$ and $Q_1,\ldots,Q_r$ are the connected components of $P$ and $Q$, respectively. Using \Cref{cor:disconnectedreg} and the fact that $\ac(Q)=\sum_{j=1}^{r}\ac(Q_i)$, we have
\begin{align*}
\reg \KK[P,Q] &= \sum_{i=1}^{s}\sum_{j=1}^{r}\reg \KK[P_i,Q_j] + (s-1) \\
&= \sum_{i=1}^{s}\sum_{j=1}^{r} (\ac(Q_j)\cdot (|P_i|-1)) + (s-1) \\
&= \ac(Q)\cdot(|P|-s)+s-1,
\end{align*}
and hence the claim follows also for disconnected $P$ and $Q$.
So, assume that both, $P$ and $Q$, are connected. 
According to \Cref{pro:unmixed}, we need to consider the following cases:\\

\noindent{\sf Case I.} $P$ is a chain.\\
In this case the result is just Corollary 3.3 in \cite{EHM}.\\

\noindent{\sf Case II.} $Q$ is a chain.\\
In this case, we have
\begin{align*}
 \reg\KK[P,Q]&=\pdim L(P,Q)^{\vee}=\pdim L(Q,P)^{\tau}=\pdim L(Q,P)\\
 &=\pdim \KK[Q,P]-1=\height(L(Q,P))-1=|P|-1.
\end{align*}
Here, the equalities in the first line hold since $L(P,Q)^{\vee}=L(Q,P)^{\tau}$, whereas the ones in the second line follow from $\KK[Q,P]$ being Cohen-Macaulay \cite[Corollary~2.5]{EHM} and from \Cref{cor:minprimes}. Since $\ac(Q)=1$, this shows the desired formula. \\

\noindent{\sf Case III.} $P$ and $Q$ are both rooted with roots $v$ and $w$, respectively.\\
Since $L(P,Q)^{\vee}=L(Q,P)^{\tau}$, we have $\Min(L(P,Q))=\{{\pc}_{\psi}~:~ \psi\in \Hom(Q,P)\}$, where ${\pc}_{\psi}=(x_{\psi(q),q}~:~q\in Q)$ (see \cite[Section 1.4]{FGH}, \cite[Section 1]{HQS}). Since $L(P,Q)$ is a squarefree monomial ideal, it is the intersection of its minimal primes. Hence $L(P,Q)=\bigcap_{\psi\in \Hom(Q,P)} \pc_{\psi}$. Let
\begin{equation}
J_1=\bigcap_{\substack{
\psi\in \Hom(Q,P) \\
\psi(w)=v
}}
{\pc}_{\psi} \quad \mbox{ and } \quad J_2=\bigcap_{\substack{
\psi\in \Hom(Q,P) \\
\psi(w)\neq v
}}{\pc}_{\psi}.
\nonumber
\end{equation}
Then $L(P,Q)=J_1\cap J_2$. We claim that $J_1=(x_{v,w})+L(P,Q\setminus w)S_{P,Q}$.
 Indeed, since each map $\phi\in \Hom(Q\setminus w,P)$ can be extended to a map in $\Hom(Q,P)$ by putting $\phi(w)=v$, we have $J_1=(x_{v,w})+\cap_{\phi\in \Hom(Q\setminus w,P)} \pc_{\phi}S_{P,Q}$. Moreover, as $P$ is connected, we have $\Min(L(P,Q\setminus w)S_{P,Q})=\{{\pc}_{\phi}S_{P,Q}~:~ \phi\in \Hom(Q\setminus w,P)\}$,  and hence $L(P,Q\setminus w)S_{P,Q}=\cap_{\phi\in \Hom(Q\setminus w,P)} \pc_{\phi}S_{P,Q}$. So, the claim follows.

 Next, we show that $J_2=L(P\setminus v,Q)S_{P,Q}$. First, note that, if $\phi\in \Hom(Q,P)$ with $\phi(w)\neq v$, then, using connectedness of $Q$, we infer that $\phi(q)\neq v$ for all $q\in Q$. Hence, $\phi\in \Hom(Q,P\setminus v)$, which implies $J_2=\cap_{\phi\in \Hom(Q,P\setminus v)} \pc_{\phi}S_{P,Q}$.
 Since $Q$ is connected, the primes $\pc_{\phi}S_{P,Q}$ in this intersection are exactly the minimal primes of $L(P\setminus v,Q)S_{P,Q}$, and hence the claim follows.

 Now we prove that $J_1+J_2=L(P\setminus v,Q)S_{P,Q}+(x_{v,w})$. For this purpose, it suffices to show that $L(P,Q\setminus w)S_{P,Q}\subseteq L(P\setminus v,Q)S_{P,Q}$. If $\phi\in \Hom(P,Q\setminus w)$, then its restriction $\tilde{\phi}$ to $P\setminus v$, is also a map in $\Hom(P\setminus v,Q)$ and thus, $u_{\tilde{\phi}}\in L(P\setminus v,Q)S_{P,Q}$. Since $u_{\tilde{\phi}}$ divides $u_{\phi}$, we conclude that $u_{\phi}\in L(P\setminus v,Q)S_{P,Q}$.

 We consider the following exact sequence of graded $S_{P,Q}$-modules
\[
0\rightarrow \KK[P,Q]\rightarrow S_{P,Q}/J_1\oplus S_{P,Q}/J_2\rightarrow S_{P,Q}/(J_1+J_2)\rightarrow 0.
\]
Then, by \cite[Corollary~18.7 (2)]{P}, we have
\begin{equation}\label{eq:reg}
\reg \KK[P,Q]\leq \max \{\reg(S_{P,Q}/J_1), \reg( S_{P,Q}/J_2) , \reg (S_{P,Q}/(J_1+J_2))+1\}.
\end{equation}
We claim that the maximum on the right-hand side of \eqref{eq:reg} is attained by $\reg (S_{P,Q}/J_1)$. \\
Indeed, since $Q$ is rooted, so is  $Q\setminus w$ and since $P$ is connected it follows from the  induction hypothesis that
\begin{equation}\label{eq1}
\reg(S_{P,Q}/J_1)=\reg(S_{P,Q\setminus w}/L(P,Q\setminus w))=\reg \KK[P,Q\setminus w] = \ac(Q)\cdot (|P|-1).
\end{equation}
where the last equality holds because $\ac(Q\setminus w)=\ac(Q)$. Similarly, since $P$ is rooted, so is $P\setminus v$. If $P\setminus v$ has $\ell$ connected components, we infer from the induction hypothesis that
\begin{equation}\label{eq2}
\reg(S_{P,Q}/J_2)=\reg \KK[P\setminus v,Q] =\ac(Q)\cdot (|P|-1-\ell)+\ell-1.
\end{equation}
Moreover, it holds that
\begin{align}
\reg(S_{P,Q}/(J_1+J_2))+1&=\reg(S_{P\setminus v,Q}/L(P\setminus v,Q))+1\notag\\
&=\reg \KK[P\setminus v,Q] +1=\ac(Q)\cdot (|P|-1-\ell)+\ell-1\label{eq3}
\end{align}
Comparing the values for the different regularities in \eqref{eq1}, \eqref{eq2} and \eqref{eq3} we see that $\reg (S_{P,Q}/J_1)$ is the maximum. Using \eqref{eq:reg} we conclude that
\begin{equation*}
\reg \KK[P,Q] \leq \reg(S_{P,Q}/J_1)=\reg \KK[P,Q\setminus w] =\ac(Q)\cdot (|P|-1).
\end{equation*}
Finally, using the lower bound in \Cref{cor:lowerboundreg}, we get
\[
\reg \KK[P,Q]=\ac(Q)\cdot (|P|-1),
\]
as desired. \\

\noindent{\sf Case IV.}  $P$ and $Q$ are both co-rooted.\\
The result follows analogously to {\sf Case III}. \\

(b) Let $P_1,\ldots,P_s$ be the connected components of $P$. Since $L(P,Q)$ is unmixed, \Cref{pro:unmixed} implies that $L(P_i,Q)^{\vee}=L(Q,P_i)^{\tau}$ for each $i=1,\ldots,s$. It follows from part (a) that
\begin{equation*}
\pdim \KK[P_i,Q] =\reg L(P_i,Q)^\vee =\reg L(Q,P_i)=\reg \KK[Q,P_i] +1=\ac(P_i)\cdot (|Q|-r)+r.
\end{equation*}
Now, by \Cref{cor:disconnectedreg}, we have
\begin{align*}
\pdim \KK[P,Q]  &= \sum_{i=1}^{s}\pdim \KK[P_i,Q]  - (s-1) \\
&= \sum_{i=1}^{s} (\ac(P_i)\cdot (|Q|-r)+r) - (s-1) \\
&= \ac(P)\cdot (|Q|-r)+sr-(s-1)=\ac(P)\cdot (|Q|-r)+s(r-1)+1,
\end{align*}
and hence the result follows.
\end{proof}

Note that the value for the Castelnuovo-Mumford regularity for the unmixed case \Cref{thm:regformula} actually equals the lower bound from \Cref{cor:lowerboundreg}. 
Moreover, we do not know any example of two posets $P$ and $Q$, such that the regularity of $L(P,Q)$ is different from the value given in \Cref{thm:regformula} (see  also \Cref{conj:bounds}). 
On the other hand, the formula for the projective dimension does not hold without the assumption of $L(P,Q)$ being unmixed:

\begin{example}\label{ex:pdim}
\begin{figure}[h!]
\begin{tikzpicture}[scale = 0.6]
\begin{scope}
	\node at (0.5,-2) {$P$};
	\draw (0,0) node[punkt] {} -- ++(-60:1) node[punkt] {} -- ++(60:1) node[punkt] {};
\end{scope}
\begin{scope}[shift={(5,0)}]
 	\node at (0.5,-2) {$Q$};
 	\draw (0,-0.9) node[punkt] {} -- ++(60:1) node[punkt] {} -- ++(-60:1) node[punkt] {};
\end{scope}
\end{tikzpicture}
%
\caption{}
\label{fig:pdim}
\end{figure}

Let $P$ and $Q$ be the posets depicted in \Cref{fig:pdim}.
 Both posets are connected but $L(P,Q)$ is not unmixed. Moreover, it holds that $\pdim \KK[P,Q] = 4$, while $\ac(P)(|Q|-1) +1 = 5$.
\end{example}

\subsection{The linear strand}
In this subsection, we adopt the following convention. Given a poset $P$, an element $p \in P$ is called \emph{maximal}, if $p \nleq p'$ for all $p' \in P$. Note that maximal elements always exist but are not necessarily unique. Minimal elements are defined analogously. If $P$ has a unique maximal and minimal element, we denote it by $\hat{1}_P$ and $\hat{0}_P$, respectively. If both, $\hat{1}_P$ and $\hat{0}_P$, exist, then $P$ is called \emph{bounded}.

\newcommand{\aco}{\overline{\ac}}
\newcommand{\acu}{\underline{\ac}}
We define $\aco(P)$ to be the number of maximal elements of $P$ and $\acu(P)$ to be  the number of minimal elements of $P$.
Note that the set of maximal elements (resp. the set of minimal elements) always forms an antichain, so it holds that $\aco(P) \leq \ac(P)$ and $\acu(P) \leq \ac(P)$. Moreover, if $P$ is rooted, then $\aco(P) = \ac(P)$. Similarly, if $P$ is co-rooted, then $\acu(P) =\ac(P)$.

Our next result is a lower bound for the length of the linear strand of $\KK[P,Q]$.
We first recall the definition of this notion. 
For a monomial ideal $I$ in a polynomial ring $S$ with $d=\min\set{\deg(u)\with u\in G(I)}$ the \emph{length of the linear strand} of $S/I$ is defined as
\[
\lin(S/I):=\max\{i~:~\beta^S_{i,i+d-1}(S/I)\neq 0\}.
\]
Moreover, we set $\lin(I):=\lin(S/I)-1$. 

\begin{proposition}\label{prop:linear}
Let $P,Q$ be posets and let $Q' \subseteq Q$ be a subposet.
\begin{enumerate}
	\item[{\em(a)}] If $Q'$ has a unique minimal element, then $\lin(\KK[P,Q]) \geq \aco(P)\cdot(|Q'|-1)+1$.
	\item[{\em(b)}] If $Q'$ has a unique maximal element, then $\lin(\KK[P,Q]) \geq \acu(P)\cdot(|Q'|-1)+1$.
	\item[{\em(c)}] If $Q'$ is bounded, then $\lin(\KK[P,Q]) \geq \ac(P)\cdot(|Q'|-1)+1$.
\end{enumerate}
\end{proposition}
\begin{proof}
By \Cref{cor:restrTarget}, we may assume that $Q = Q'$.

We first consider the case that $Q$ is bounded.
Let $A \subseteq P$ be a maximal antichain. Consider the sets
\begin{align*}
P_{>A} &:= \set{p \in P\setminus A \with p > a \text{ for some } a \in A}\qquad \text{and}\\
P_{<A} &:= \set{p \in P\setminus A \with p < a \text{ for some } a \in A}.
\end{align*}
As $A$ is maximal, $P = P_{>A} \cup A \cup P_{<A}$ is a disjoint decomposition of $P$.
Let $H \subseteq \Hom(P,Q)$ be the set of homomorphisms $\phi:P\to Q$, which send $P_{>A}$ to $\hat{1}_Q$ and $P_{<A}$ to $\hat{0}_Q$.
Define
\[ m_A := \prod_{p \in P_{<A}} x_{p,\hat{0}_Q} \cdot \prod_{p \in P_{>A}} x_{p,\hat{1}_Q}. \]
Then the monomials $u_\phi$ with $\phi \in H$ are exactly the monomials of the form $m_A \cdot u_\psi$ for some $\psi \in \Hom(A,Q)$.
So the restriction of $L(P,Q)$ to the variables occurring in generators corresponding to the elements of $H$ equals $m_A L(A,Q)$.
Hence
\newcommand{\mb}{\mathbf{m}}
\[ \beta_{i,\ab + \mb}(\KK[P,Q]) \geq \beta_{i,\ab}(\KK[A,Q]) \]
for any $i$ and $\ab \in \NN^{A \times Q}$ and $\mb = \deg m_A$.
In particular, $\beta_{p,p+|P|-1}(\KK[P,Q]) \geq \beta_{p,p+|A|-1}(\KK[A,Q]) > 0$
for $p = |A|\cdot(|Q|-1)+1$ by \Cref{ex:antichain} (a) and the claim follows.

If $Q$ has only a unique minimal element, then we apply the same argument to the antichain $A \subset P$ consisting of all maximal elements. In this case $P_{>A} = \emptyset$, so we do not need to refer to $\hat{1}_Q$. The case that $Q$ has only a unique maximal element is analogous.
\end{proof}

In some special situations, the preceding result shows that the length of the linear strand equals the projective dimension:
\begin{corollary}\label{cor:linear}
	Let $P$ and $Q$ be posets.
\begin{enumerate}
	\item[\emph{(a)}] If $Q$ is connected and $P$ and $Q$ are both rooted or both co-rooted, then
	\[ \pdim \KK[P,Q] = \lin(\KK[P,Q]) = \ac(P)\cdot(|Q|-1)+1. \]
	\item[\emph{(b)}] If $P$ is a chain and $Q$ has a unique minimal or a unique maximal element, then
	\[ \pdim \KK[P,Q] = \lin(\KK[P,Q]) = |Q|. \]
\end{enumerate}
\end{corollary}
\begin{proof}
\begin{asparaenum}
\item[(a)] If $P$ is rooted, then $\aco(P) = \ac(P)$, and if $Q$ is connected and rooted, then it has a unique minimal element, so
	$\pdim \KK[P,Q] \geq \lin(\KK[P,Q]) \geq \ac(P)\cdot(|Q|-1)+1$ by the preceding proposition.
	On the other hand, $\pdim \KK[P,Q] = \ac(P)\cdot(|Q|-1)+1$ by \Cref{thm:regformula}.
	The co-rooted case is proven similarly.
\item[(b)] If $P$ is a chain, then $\aco(P) = \acu(P) = 1$, so \Cref{prop:linear} implies that $\lin(\KK[P,Q]) \geq |Q|$.
	On the other hand, $L(P,Q)$ is unmixed by \Cref{pro:unmixed} and both $P$ and $Q$ are connected, thus $\pdim \KK[P,Q] = |Q|$  by \Cref{thm:regformula}.
\end{asparaenum}
\end{proof}

In many cases, the bound given by \Cref{prop:linear} gives the actual length of the linear strand for a suitable choice of $Q'$. However, this does not always hold as the following example shows:
\begin{example}\label{ex:cross}
\begin{figure}[h!]
\begin{tikzpicture}[scale = 0.6]
\begin{scope}
	\node at (0,-2) {$P$};
	\draw (0,-0.5) node[punkt] {} -- ++(0,1) node[punkt] {};
\end{scope}
\begin{scope}[shift={(5,0)}]
 	\node at (0,-2) {$Q$};
 		\draw (0,0) node[punkt] {} -- (60:1) node[punkt] {}
		(0,0) -- (-60:1) node[punkt] {}
		(0,0) -- (120:1) node[punkt] {}
		(0,0) -- (-120:1) node[punkt] {};
\end{scope}
\end{tikzpicture}
\caption{}
\label{fig:cross}
\end{figure}
Consider $L(P,Q)$ for the two posets depicted in \Cref{fig:cross}.
One can easily compute with \texttt{Macaulay2} \cite{M2} that $\lin \KK[P,Q]=\pdim \KK[P,Q]=5$,
while the best bound obtainable from \Cref{prop:linear} is $4$.

\end{example}

In \cite{FGH}, letterplace and co-letterplace ideals of posets are studied and it is shown that several well-known classes of ideals arise as quotients of letterplace ideals by a regular sequence. This also provides part of the motivation for this article since the results, obtained so far, can be applied to those classes of ideals. 

\begin{example}
\begin{asparaenum}
\item[(a)] {\sf Initial ideals of determinantal ideals:}\\
It was shown in \cite[Section 3.2 and 3.3]{FGH} that the initial ideal of $s$-minors of a general $(n+s-1)\times (m+s-1)$-matrix $X$ with respect to a diagonal term order can be realized by cutting down $L([s],[m]\times [n])$ by a regular sequence. Let $I_s(X)$ denote the ideal of $s$-minors of $X$ and $\init_{\prec}(I_s(X))$ its initial ideal with respect to a diagonal term order.
Since $[m]\times [n]$ is bounded, it directly follows from \Cref{cor:linear}, that 
\begin{equation*}
\pdim \init_{\prec}(I_s(X)) = \lin( \init_{\prec}(I_s(X))) =  m\cdot n -1.
\end{equation*}
So as to compute the regularity of $\init_{\prec}(I_s(X))$, it follows from \Cref{thm:regformula} or \cite[Corollary 3.3]{EHM} that we need to determine $\ac([m]\times [n]) $. For this, first note that since $[n]$ and $[m]$ have the Sperner property (see \cite{StTopicsInAlg} for the definition of the Sperner property), so has $[m]\times [n]$. In particular, the size of a maximal antichain in $[m]\times [n]$ equals the maximal size of a set of elements of the same rank. Moreover, since $[m]\times [n]$ is rank-symmetric and since its  rank numbers form a unimodal sequence, it follows that
 \begin{equation*}
  \ac([m]\times [n])=|\{(p,q)\in [m]\times [n]~:~p+q=\lfloor\frac{m+n+2}{2}\rfloor\}|=\min(m,n).
\end{equation*}
This yields $\reg \init_{\prec}(I_s(X)) =\min(m,n)\cdot (s-1)+1.$

\item[(b)] {\sf Initial ideals of the ideals of \texorpdfstring{$2$}{2}-minors of a symmetric matrix:}\\
We briefly recall a construction from \cite{FGH}. Given posets $P$ and $Q$, the set $\Hom(P,Q)$ of order-preserving maps $\phi:P\rightarrow Q$ is turned into a poset by setting $\phi\leq \psi$ if $\phi(p)\leq \psi(p)$ for all $p\in P$. 
Reducing the ideal $L([2],\Hom([2],[n]))$ via a particular regular sequence, one obtains the ideal generated by $x_{i_1,i_2}x_{j_1,j_2}$ with $i_1<j_1$, $i_2<j_2$ and $i_1\leq i_2$ and $j_1\leq j_2$. This ideal is an initial ideal of the ideal of $2$-minors of a general symmetric $(n+1)\times (n+1)$-matrix (see \cite[Section 3.4]{FGH} for more details, and also \cite[Section 5]{CHT}). In the following, we denote this ideal by $I$. \Cref{cor:linear} implies that 
\begin{equation*}
\pdim I = \lin(I) = |\Hom([2],[n])|-1=\frac{n(n+1)}{2}-1.
\end{equation*}
Moreover, by \Cref{thm:regformula} or \cite[Corollary 3.3]{EHM}, we have $\reg I=\ac(\Hom([2],[n]))+1 $. 
To compute $\ac(\Hom([2],[n]))$, note that every antichain $\phi_1, \phi_2, \dots, \phi_s$ in $\Hom([2],[n])$ can be ordered such that
\[ \phi_1(1) < \phi_2(1) < \dotsb < \phi_s(1) \leq \phi_s(2) < \phi_{s-1}(2) < \dotsb < \phi_1(2). \]
Thus, $\ac(\Hom([2],[n]))= \lceil\frac{n}{2}\rceil$ and hence $\reg I= \lceil\frac{n}{2}\rceil+1$.
\end{asparaenum}
\end{example}

\section{From Buchsbaum-ness to Gorenstein-ness}\label{sect:charact}
In this section, we characterize when $\KK[P,Q]$ has a certain algebraic property in terms of the posets $P$ and $Q$.
The considered properties include 
Gorenstein-ness of $\KK[P,Q]$, as well as (sequentially) Cohen-Macaulay-ness, Buchsbaum-ness and satisfying Serre's condition $(S_r)$.

Let $S=K[x_1,\ldots,x_n]$ be a polynomial ring and let $I$ be a monomial ideal in $S$, which is generated in a single degree. Recall that $I$ is called \emph{weakly polymatroidal} if for any two monomials $u={x_1}^{a_1}\cdots {x_n}^{a_n}$ and $v={x_1}^{b_1}\cdots {x_n}^{b_n}$ in $G(I)$ for which there exists an integer $t$ with $a_1 = b_1,\ldots, a_{t-1} = b_{t-1}$ and $a_t > b_t$, there exists $l>t$ such that $x_t(v/x_l)\in I$. Note that here the order of the variables of $S$ is considered as $x_1>\cdots>x_n$.

Moreover, an ideal $I\subseteq S$ is said to have \emph{linear quotients} if one can order the minimal generators of $I = (m_1, \dotsc, m_r)$, such that the colon ideals $(m_1, \dots, m_{i-1})\colon m_i$ are generated by linear forms for $i = 2, \dotsc, r$. It was shown in \cite[Theorem~1.4]{KH} that a weakly polymatroidal ideal has linear quotients.


\Polymatroidal
\begin{proof}
The implications ``(a)$\Rightarrow$(b)'', ``(b)$\Rightarrow$(c)'' and ``(c)$\Rightarrow$(d)'' are well known.
\begin{asparaenum}
\item[``(d)$\Rightarrow$(e)'':]

	Assume to the contrary that neither $P$ is an antichain nor $Q$ is a chain.
	Then $P$ has a connected component $P_1 \subseteq P$ with at least two elements, and $Q$ has an antichain of length $2$.
	Thus, by \Cref{cor:lowerboundreg}, there exists a multidegree $\ab \in \NN^{P_1 \times Q}$ with total degree $2|P_1|$ such that $\beta_{2,\ab}(\KK[P_1,Q]) \neq 0$.
	If $P = P_1$, then this contradicts our assumption (d).
	Otherwise, let $P' := P \setminus P_1$ and let $\bb \in \NN^{P'\times Q}$ be the multidegree of any generator of $L(P',Q)$.
	\Cref{prop:disjoint} implies that
	\[ \beta_{2,(\ab,\bb)}(\KK[P,Q]) \geq \beta_{2,\ab}(\KK[P_1,Q])\cdot\beta_{1,\bb}(\KK[P',Q]) \geq 1. \]
	The total degree of $(\ab,\bb)$ is $2 |P_1| + |P'| = |P| + |P_1| \geq |P| + 2$, which 
	again contradicts (d).

\item[``(e)$\Rightarrow$(a)'':] If $Q$ is a chain, then $L(Q,P)^\vee=L(P,Q)^{\tau}$ by \Cref{prop:dualclass} and it follows from \cite[Theorem~2.2]{EHM} that all powers of $L(P,Q)$ are weakly polymatroidal.

		Now, assume that $P$ is an antichain and let $k\geq 1$.
		Suppose that $P=\{p_1,\ldots,p_r\}$ and $Q=\{q_1,\ldots,q_l\}$.
		We order the variables of $S_{P,Q}$ lexicographically, namely $x_{p_i,q_j}>x_{p_{i'},q_{j'}}$ if and only if $i<i'$, or $i=i'$ and $j<j'$.
		Obviously, ${L(P,Q)}^k$ is generated in the single degree $k |P|$.
		
		Let $\mathbf{u}$ and $\mathbf{v} =u_{\phi_1}\ldots u_{\phi_k}$ be two monomial generators of ${L(P,Q)}^k$,
		where  $\phi_i\in \Hom(P,Q)$ for all $i=1,\ldots,k$ (which are not necessarily pairwise distinct).
		To simplify notation, we denote by $\deg_{i,j} \bf{w}$ the degree of the variable $x_{p_i,q_j}$ in any monomial $\bf{w}$.
		Assume that ${\deg}_{t,s}{\bf{u}} > {\deg}_{t,s}{\bf{v}}$ for some $t$ and $s$, and $\deg_{i,j}{\bf{u}}={\deg}_{i,j}{\bf{v}}$ for all $x_{p_i,q_j}> x_{p_t,q_s}$.
		Note that $\sum_{i=1}^l \deg_{t,i} \mathbf{w} = k$ for any $\mathbf{w} \in G(L(P,Q)^k)$ and $t=1,\ldots,r$.
		Hence we have that $\sum_{i=1}^s \deg_{t,i} \mathbf{v} < \sum_{i=1}^s \deg_{t,i} \mathbf{u} \leq k$ and thus $\sum_{i=s+1}^l \deg_{t,i} \mathbf{v} > 0$.
		So there exists an $i$ such that $\phi_i(p_t) = q_h$ for some $h > s$; and in particular $s < l$.
		We define ${\widetilde{\phi}}_{i}:P \to Q$ as follows:
		\[ \widetilde{\phi}_i(p) = \begin{cases}
		q_s &\text{ if } p = p_t;\\
		\phi_i(p) &\text{ otherwise.}
		\end{cases}\]
		Since $P$ is an antichain, we have ${\widetilde{\phi}}_{i}\in \Hom(P,Q)$, and hence $u_{{\widetilde{\phi}}_{i}}\in G(L(P,Q))$.
		Thus, $u_{{\widetilde{\phi}}_{i}}({\bf{v}}/u_{\phi_i})\in {L(P,Q)}^k$.
		On the other hand, $u_{\widetilde{\phi}_{i}}({\bf{v}}/u_{\phi_i})=x_{p_t,q_s}({\bf{v}}/x_{p_t,q_h})$, and hence $x_{p_t,q_s}({\bf{v}}/x_{p_t,q_h})\in  {L(P,Q)}^k$. As  $x_{p_t,q_s} > x_{p_t,q_h}$ because of $h>s$, we conclude that ${L(P,Q)}^k$ is weakly polymatroidal.
\end{asparaenum}

\end{proof}

Note that the condition on Betti numbers of an ideal $I$ given in \Cref{cor:charlinres} part (d), is sometimes called \emph{linearly related} (see for example \cite{EHH}) or \emph{$N_{|P|,2}$ property} (see for example \cite{EGHP}). 

\medskip
In the following, we classify those pairs of posets $(P,Q)$ such that $\KK[P,Q]$ is Cohen-Macaulay or satisfies Serre's condition $(S_r)$ for an $r\geq 2$.
We first recall the definition of Cohen-Macaulayness and Serre's condition $(S_r)$. Given a Noetherian graded $\KK$-algebra $R$ and a finitely generated graded $R$-module $M$, we say that $M$ is \emph{Cohen-Macaulay} if $\depth M = \dim M$. We say that $M$ satisfies \emph{Serre's condition $(S_r)$} if $\depth(M_{\mathfrak{p}})\geq \min\{r,\height(\mathfrak{p})\}$ for all $\mathfrak{p}\in \Spec(R)$.

Recall that a graded $R$-module $M$ is called \emph{sequentially Cohen-Macaulay} if there exists a finite filtration of graded $R$-modules
$0=M_{0}\subset M_{1}\subset\cdots \subset M_{r}=M$ such that each $M_{i}/M_{i-1}$ is Cohen-Macaulay, and the Krull dimensions of the quotients
are increasing, i.e.,
\[
\mathrm {dim} (M_{1}/M_{0}) < \mathrm{dim}(M_{2}/M_{1}) < \cdots < \mathrm{dim}(M_{r}/M_{r-1}).
\]
It is well-known that a Stanley-Reisner ring $\KK[\Delta]$ is sequentially Cohen-Macaulay if and only if the pure $i$-skeleta  $\Delta^{[i]}$ of $\Delta$ are Cohen-Macaulay for all $-1\leq i \leq \dim \Delta$, cf. \cite[Theorem 3.3]{D}.
Here, the \emph{pure $i$-skeleton} of a simplicial complex $\Delta$ is the subcomplex $\Delta^{[i]} \subseteq \Delta$, whose facets are all $i$-dimensional faces of $\Delta$.

\CohenMacaulay
\begin{proof}
The implications ``$(a)\Rightarrow(b)\Rightarrow(c)$'' and ``$(b)\Rightarrow(d)$'' are either well-known or trivial.

\begin{asparaenum}
\item[``$(c)\Rightarrow(e)$'' and ``$(d)\Rightarrow(e)$'':]
Recall from \cite{HTYZ} that $\KK[P,Q]$ is called \emph{sequentially $(S_2)$} if every quotient in the dimension filtration of $\KK[P,Q]$ satisfies Serre's condition $(S_2)$. Clearly, both $(c)$ and $(d)$ imply that $\KK[P,Q]$ is sequentially $(S_2)$.
So it is sufficient to show that this implies $(e)$.

We first show that $P$ is connected.
So assume on the contrary that $P$ is not connected, i.e., it is a disjoint union $P=P_1 \cup P_2$ of two subposets (which are not necessarily connected themselves).
Let $\Delta_i := \Delta(P_i,Q)$ and let $V_i$ be the set of vertices of $\Delta_i$ for $i = 1,2$. To simplify notation, we set $\Delta = \Delta(P,Q)$. If follows from $L(P,Q) = L(P_1,Q)L(P_2,Q)$ that
\begin{equation*} 
\Delta = \set{F_1 \cup F_2 \with F_1 \subseteq V_1, F_2 \subseteq V_2 \text{ and } F_1 \in \Delta_1\text{ or } F_2 \in \Delta_2}.
\end{equation*}
By \cite[Theorem 2.6]{HTYZ} $\Delta$ is sequentially $(S_2)$ if and only if its pure $i$-skeleton $\Delta^{[i]}$ is $(S_2)$ for all $-1\leq i\leq \dim\Delta$.
Further, it follows from Hochster's formula for the local cohomology that this is the case if and only if the link of every face $F \in \Delta^{[i]}$ with $\dim F \leq \dim \Delta^{[i]}-2$ is connected for $-1\leq i\leq \dim \Delta$ (see also \cite[page 4]{Te}).
Let $F_1 \in \Delta_1, F_2 \in \Delta_2$ be facets and assume that $|V_1\setminus F_1| \leq |V_2\setminus F_2|$.
Then $F := F_1 \cup F_2 \in \Delta$ and $\lk_\Delta F = 2^{V_1\setminus F_1} \cup 2^{V_2\setminus F_2}$.
Even more, for $i = \dim F + |V_1\setminus F_1|$ we have that $\lk_{\Delta^{[i]}} F = 2^{V_1\setminus F_1} \cup (2^{V_2\setminus F_2})^{[i]}$, which is not connected.
On the other hand, \Cref{cor:minprimes} implies that $|V_1\setminus F_1| \geq |Q|\geq 2$, and hence
\[\dim F + 2\leq \dim F + |V_1\setminus F_1| = \dim\Delta^{[i]}.\]
So we arrive at a contradiction, and hence $P$ is connected.

By \cite[Theorem 3.2]{HTYZ}, $\KK[P,Q]$ being sequentially $(S_2)$ implies that the first two steps in the minimal graded free resolution of $L(P,Q)^\vee$ are componentwise linear.
We are going to show that the latter condition is violated if $P$ is not a chain and $Q$ is not an antichain.
So assume to the contrary that $P$ contains two incomparable elements $p_1, p_2$ and $Q$ has a connected component $Q' \subseteq Q$ of cardinality at least $2$.
Let $\psi_1 \in \Hom(Q,P)$ be the constant map with value $p_1$, and let $\psi_2 \in \Hom(Q,P)$ be the map which maps $Q'$ to $p_2$ and $Q \setminus Q'$ to $p_1$.
It is easy to see that both maps satisfy the condition of part (b) of \Cref{cor:minprimes}, hence $\pp_{\psi_1}, \pp_{\psi_2} \in \Min(L(P,Q))$ (for $\psi_2$, one needs that $P$ is connected).
Thus, $u_{\psi_1}$ and $u_{\psi_2}$ are minimal generators of $L(P,Q)^\vee$.
Moreover, every generator of $L(P,Q)^\vee$ of degree $|Q|$ is of the form $u_\psi$ for some $\psi \in \Hom(Q,P)$ by \Cref{prop:minprimes}.
It is not difficult to see that there is no map $\psi \in \Hom(Q,P), \psi \neq \psi_1, \psi_2$, such that $u_\psi$ divides the least common multiple $u$ of $u_{\psi_1}$ and $u_{\psi_2}$, because every $\psi$ whose image is contained in $\set{p_1,p_2}$ maps $Q'$ completely either to $p_1$ or to $p_2$.
Hence, the degree $|Q|$-part of $L(P,Q)^\vee$ has a syzygy in the degree of $u$, which is $|Q| + |Q'| > |Q| + 1$.
Thus the second step of the resolution of $L(P,Q)^\vee$ is not componentwise linear, a contradiction.


\item[``$(e)\Rightarrow(a)$'':] In both cases we have that $L(P,Q)^\vee = L(Q,P)^\tau$.
	Hence \Cref{cor:charlinres} implies that $L(P,Q)^\vee$ has linear quotients and thus $\Delta(P,Q)$ is shellable.
\end{asparaenum}
\end{proof}

\begin{remark}\label{rem:LPcase}
	In \cite{EHM} it was shown that $L(P,Q)$ has a linear resolution if $Q$ is a chain.
	The preceding \Cref{cor:charlinres} can be seen as a converse: If $L(P,Q)$ has a linear resolution, then either $Q$ is a chain, or $P$ is an antichain. But in the latter case, $L(P,Q)$ only depends on the size of $Q$, and so $L(P, Q) = L(P, Q')$ where $Q'$ is a chain with $|Q'| = |Q|$.
	
	Similarly, \Cref{cor:CM} gives a converse to the result that $\KK[P,Q]$ is Cohen-Macaulay if $P$ is a chain, cf. \cite{EHM, FGH}. If $\KK[P,Q]$ is Cohen-Macaulay, then $P$ is connected and either $P$ is a chain or $Q$ is an antichain. In the latter case, $\Hom(P,Q)$ contains only the constant maps, and thus $\KK[P,Q] \cong \KK[P', Q]$ where $P'$ is the chain with $|P|$ elements.
\end{remark}

The next corollary provides a characterization of all pairs $(P,Q)$, for which $\KK[P,Q]$ is Gorenstein or a complete intersection. Let $I$ be a monomial ideal in a polynomial ring $S$ over a field $\KK$. Then, recall that $S/I$ is \emph{Gorenstein} if and only if $S/I$ is Cohen-Macaulay
and the last total Betti number in the minimal graded free resolution of $S/I$ is equal to $1$ (see \cite[Corollary 4.3.5]{Vi}).
Furthermore, $S/I$ is a \emph{complete intersection} if the cardinality of $I$ equals its height, i.e., $|G(I)|=\height(I)$.  
Equivalently, one can see that $S/I$ is a complete intersection if and only if any two elements of $G(I)$ are on disjoint sets of variables.

\begin{corollary}\label{cor:Gor}
Let $P$ and $Q$ be posets such that $|Q|\geq 2$. Then the following conditions are equivalent:
 \begin{itemize}
  \item[{\em(a)}] $\KK[P,Q]$ is a complete intersection.
  \item[{\em(b)}] $\KK[P,Q]$ is Gorenstein.
  \item[{\em(c)}] $P$ is connected and $Q$ is an antichain.
 \end{itemize}
\end{corollary}

\begin{proof}
The implication ``$(a)\Rightarrow(b)$'' is well-known, and ``$(c)\Rightarrow(a)$'' follows from \Cref{ex:antichain} $(b)$.

\begin{asparaenum}
\item[``$(b)\Rightarrow(c)$'':] Suppose that $\KK[P,Q]$ is Gorenstein. Then it is Cohen-Macaulay, and hence it follows from \Cref{cor:CM} that $P$ is connected, and $P$ is a chain or $Q$ is an antichain. In the first case, Gorenstein-ness of $\KK[P,Q]$ implies that $Q$ is an antichain (\cite[Corollary~2.5]{EHM}), and the claim follows.
\end{asparaenum}
\end{proof}

In the following, we study the locally Cohen-Macaulay property and the Buchsbaum property of the rings $\KK[P,Q]$. We first recall the relevant definitions.
A Stanley-Reisner ring  $\KK[\Delta]$ is called \emph{locally Cohen-Macaulay} if every localization is Cohen-Macaulay,
except possibly the localization at the irrelevant ideal.
Further, $\KK[\Delta]$ is \emph{Buchsbaum} if and only if it is locally Cohen-Macaulay and unmixed.
Note that the latter is not the original definition but an equivalent characterization due to Miyazaki \cite[Theorem~2 (iv)]{Miya}.

\begin{corollary}\label{cor:Buchsbaum}
	Let $P$ and $Q$ be posets such that $|Q|\geq 2$. Then the following conditions are equivalent:
	\begin{enumerate}
	\item[{\em(a)}] $\KK[P,Q]$ is Buchsbaum, but not Cohen-Macaulay.
	\item[{\em(b)}] $\KK[P,Q]$ is locally Cohen-Macaulay, but not Cohen-Macaulay.
	\item[{\em(c)}] $P$ is an antichain with $2$ elements.
	\end{enumerate}
\end{corollary}
\begin{proof}
The implication ``$(a)\Rightarrow(b)$'' is clear.
\begin{asparaenum}
\item[``$(b)\Rightarrow(c)$'':] If $\KK[P,Q]$ is not Cohen-Macaulay, then, by \Cref{cor:CM}, it does not satisfy $(S_2)$.
But as it is locally Cohen-Macaulay, its depth is at most $1$, equivalently its projective dimension is at least $|P|\cdot|Q|-1$.
Now \Cref{cor:easybound} implies that $|P|\cdot |Q|-1 \leq |P|(|Q|-1)+1$ and hence $|P| \leq 2$.
Finally, $P$ needs to be an antichain consisting of $2$ elements, since otherwise $\KK[P,Q]$ would be Cohen-Macaulay by \Cref{cor:CM}.

\item[``$(c)\Rightarrow(a)$'':]
If  $P$ is a $2$-element antichain, then $L(P,Q)$ is the product of two ideals, each generated by $|Q|$ variables.
It is not difficult to see that this implies that $\Delta(P,Q)$ is the disjoint union of two $(|Q|-1)$-simplices.
Thus $\KK[P,Q]$ is unmixed and locally Cohen-Macaulay, i.e., Buchsbaum.
\end{asparaenum}
\end{proof}

Finally, we consider the Golod property.
This is a rather subtle homological property, so we refrain from giving a full definition.
Instead we refer the reader to the survey article \cite{Avra} by Avramov or to \cite{GL}.
As it is well-known that $S/I$ is Golod whenever $I$ has a linear resolution (cf. \cite{HRW}), the following result should be compared with \Cref{cor:charlinres}.
\begin{proposition}\label{prop:golod}
%
Let $P$ and $Q$ be posets. Then the following are equivalent:
\begin{enumerate}[\em (a)]
	\item $\KK[P,Q]$ is Golod.
	\item $P$ is disconnected or $Q$ is a chain.
\end{enumerate}
\end{proposition}
\begin{proof}
\begin{asparaenum}
\item[``$(a)\Rightarrow(b)$'':]
	We proceed by proof by contradiction.
	Assume that $\KK[P,Q]$ is Golod, $P$ is connected and $Q$ is not a chain.
	If $\KK[P,Q]$ is Golod, then by \cite[Lemma 7.4]{Joe}, $L(P,Q)$ satisfies the \emph{gcd-condition}, i.e., for any two coprime monomial generators $u_1, u_2 \in G(L(P,Q))$ 
	there is a third monomial generator $u \neq u_{1}, u_{2}$ of $L(P,Q)$, which divides the least common multiple of $u_1$ and $u_2$.
	Now, if $Q$ is not a chain then it contains two incomparable elements $q_1, q_2$.
	The two constant maps $\phi_1, \phi_2: P \to Q$ with image $q_1$ and $q_2$, respectively, correspond to two coprime generators $u_{\phi_1}$ and $u_{\phi_2}$ of $L(P,Q)$.
	But $P$ being connected implies (as in the proof of \Cref{cor:CM}), that there is no monomial generator $u \neq u_{\phi_1}, u_{\phi_2}$ dividing their least common multiple.
	So we arrive at a contradiction.
	%

\item[``$(b)\Rightarrow(a)$'':]
	If $Q$ is a chain, then $\KK[P,Q]$ has a linear resolution and this implies that $\KK[P,Q]$ is Golod, cf.~\cite{HRW}.
	
	Moreover, if $P$ is disconnected with connected components $P_1, \dotsc, P_s$, then $L(P,Q)$ is the product of the ideals $L(P_i, Q)$.
	These ideals live in different variables, hence their product equals their intersection.
	Under this assumption, Herzog and Steurich showed in \cite{HS} that $\KK[P,Q]$ is Golod.
\end{asparaenum}
\end{proof}

It was proven by Seyed Fakhari and Welker in \cite{SFW} that arbitrary products of monomial ideals are Golod, but recently de Stefani found a counterexample to this result \cite{dSt}.
In fact,  as it was  remarked in \cite{dSt} the proofs in \cite{SFW} are correct and the problem lies in \cite{Joe}, on which the former article is based.


\section{Open problems}\label{sect:questions}

In this section, we pose a conjecture regarding formulas for some invariants of the ideals $L(P,Q)$.
We also discuss some problems around the topic which have not been answered yet.

\subsection{The regularity and the projective dimension in the general case}\label{invariants}
We would like to generalize \Cref{thm:regformula} and propose the following conjecture.
\begin{conjecture}[Castelnuovo-Mumford regularity and projective dimension]\label{conj:bounds}
Let $P$ and $Q$ be posets with $s$ and $t$ connected components, respectively. Then
\begin{enumerate}
	\item[(a)] $\reg(\KK[P,Q])=\ac(Q)\cdot (|P|-s)+s-1$.
	\item[(b)] $\pdim(\KK[P,Q]) \leq \ac(P)\cdot (|Q|-t)+s(t-1)+1$.
\end{enumerate}
\end{conjecture}
Note that this conjecture is true if $L(P,Q)$ is unmixed by \Cref{thm:regformula}.
In addition, it follows from \Cref{cor:lowerboundreg} that the right-hand side in (a) is a lower bound for the regularity in general. 
On the other hand, as  \Cref{ex:pdim} shows, the inequality in part (b) can be strict.


\subsection{The Level property}\label{level}

\Cref{cor:CM} and \Cref{cor:Gor} provide characterizations of all pairs of posets $(P,Q)$ for which $\KK[P,Q]$ is Cohen-Macaulay and Gorenstein, respectively.
A property that is stronger than Cohen-Macaulay-ness but weaker than the Gorenstein property is being level.
Recall that for a homogeneous ideal $I$ in a polynomial ring $S$, $S/I$ is \emph{level} if and only if $S/I$ is Cohen-Macaulay and the last step of its minimal graded free resolution is pure, namely $\beta_{p}(S/I)=\beta_{p,j}(S/I)$, for some $j$, where $p=\pdim S/I$, cf. \cite[Chapter III.3]{StaCCA}.
We would like to suggest the following problem.
\begin{problem}
	Characterize all pairs of posets $(P,Q)$ such that $\KK[P,Q]$ is level%
	\footnote{After the present paper appeared on the arXiv, this problem was solved by D'Al{\`i}, Fl{\o}ystad and Nematbakhsh in \cite{AFN}.}.
\end{problem}

We now give some possible directions concerning this problem.
By \Cref{cor:Gor}, $\KK[P,Q]$ is level if $P$ is connected and $Q$ is an antichain.
Using \Cref{cor:CM} it hence remains to classify those posets $Q$ such that $\KK[[n],Q]$ is level.
The following example shows that the classes of posets such that $\KK[P,Q]$ is Gorenstein, level and Cohen-Macaulay, respectively, differ.

\begin{example}
\begin{asparaenum}[(a)]

\begin{figure}[h!]
\begin{tikzpicture}[scale = 0.6]
	\begin{scope}[shift={(5,0)}]
		\node at (1.1,-1) {$Q_2$};
		\draw (0,1) node[punkt] {} -- (0.5,0) node[punkt] {} -- (1,1) node[punkt] {} -- (1.5,0) node[punkt] {} -- (2,1) node[punkt] {};
	\end{scope}
	\begin{scope}
		\node at (0.6,-1) {$Q_1$};
		\draw (0,0) node[punkt] {} -- (0,1) node[punkt] {} -- (1,0) node[punkt] {} -- (1,1) node[punkt] {};
	\end{scope}
\end{tikzpicture}
\caption{}
\label{figLevel}
\end{figure}
\item Let $P = [2]$ and let $Q_1, Q_2$ be the two posets in \Cref{figLevel}.
Then $\KK[P,Q_1]$ is level, but $\KK[P, Q_2]$ is not. This can be verified directly with \texttt{Macaulay2}.
\item  More generally, if $P$ is a chain and $Q$ is bounded, then $\KK[P,Q]$ is not level unless it has a linear resolution.
  To see this, note that by \Cref{prop:linear} the linear strand in the resolution of $\KK[P,Q]$ continues until the projective dimension of $\KK[P,Q]$. On the other hand, $\KK[P,Q]$ being Cohen-Macaulay implies that the strand of the resolution with the maximal shift also extends up to the projective dimension, cf. \cite[Proposition~4.2.3]{Vi}. Hence, $\KK[P,Q]$ is not level, unless the resolution is linear.

\item Further, $\KK[P,Q]$ is level if $P$ is a chain and $Q$ is a disjoint union of chains.
This follows from \Cref{cor:disconnectedreg} and \Cref{cor:Gor}.
\end{asparaenum}
\end{example}

\bibliographystyle{amsalpha}
\bibliography{Lit}

\providecommand{\bysame}{\leavevmode\hbox to3em{\hrulefill}\thinspace}
\providecommand{\MR}{\relax\ifhmode\unskip\space\fi MR }
\providecommand{\MRhref}[2]{%
  \href{http://www.ams.org/mathscinet-getitem?mr=#1}{#2}
}
\providecommand{\href}[2]{#2}
\begin{thebibliography}{HTYZN11}

\bibitem[Avr98]{Avra}
Luchezar~L. Avramov, \emph{Infinite free resolutions}, Six lectures on
  commutative algebra, Springer, 1998, pp.~1--118.

\bibitem[CHT06]{CHT}
Aldo {Conca}, Serkan {Ho\c sten}, and Rekha~R. {Thomas}, \emph{{Nice initial
  complexes of some classical ideals.}}, Algebraic and {G}eometric
  {C}ombinatorics. Euroconference in {M}athematics, AMS, 2006, pp.~11--42.

\bibitem[DFN16]{AFN}
Alessio D'Al{\`i}, Gunnar Fl{\o}ystad, and Amin Nematbakhsh, \emph{Resolutions
  of letterplace ideals of posets}, arXiv preprint arXiv:1601.02792 (2016).

\bibitem[Duv96]{D}
Art~M. Duval, \emph{Algebraic shifting and sequentially {C}ohen-{M}acaulay
  simplicial complexes}, Electron. J. Combin \textbf{3} (1996), no.~1, R21.

\bibitem[EGHP05]{EGHP}
David Eisenbud, Mark Green, Klaus Hulek, and Sorin Popescu, \emph{Restricting
  linear syzygies: algebra and geometry}, Compositio Math. \textbf{141} (2005),
  1460--1478.

\bibitem[EHH15]{EHH}
Viviana Ene, J{\"u}rgen Herzog, and Takayuki Hibi, \emph{Linearly related
  polyominoes}, Journal of Algebraic Combinatorics \textbf{41} (2015), no.~4,
  949--968.

\bibitem[EHM11]{EHM}
Viviana Ene, J{\"u}rgen Herzog, and Fatemeh Mohammadi, \emph{Monomial ideals
  and toric rings of {H}ibi type arising from a finite poset}, European Journal
  of Combinatorics \textbf{32} (2011), no.~3, 404--421.

\bibitem[Eis95]{Ei}
David Eisenbud, \emph{Commutative algebra with a view toward algebraic
  geometry}, Springer, 1995.

\bibitem[FGH15]{FGH}
Gunnar Fl{\o}ystad, Bj{\o}rn~M{\o}ller Greve, and J{\"u}rgen Herzog,
  \emph{Letterplace and co-letterplace ideals of posets}, arXiv preprint
  arXiv:1501.04523 (2015).

\bibitem[GL69]{GL}
Tor~H. Gulliksen and Gerson Levin, \emph{Homology of local rings}, Queen's
  Paper in Pure and Applied Mathematics, vol.~20, Queen's University, Kingston,
  Ont., 1969.

\bibitem[GPW99]{GPW}
Vesselin Gasharov, Irena Peeva, and Volkmar Welker, \emph{The lcm-lattice in
  monomial resolutions}, Mathematical Research Letters \textbf{6} (1999),
  no.~5, 521--532.

\bibitem[GS]{M2}
Daniel~R. Grayson and Michael~E. Stillman, \emph{Macaulay2, a software system
  for research in algebraic geometry}, Available at
  http://www.math.uiuc.edu/Macaulay2/.

\bibitem[Hat01]{Hat}
Allen Hatcher, \emph{Algebraic topology}, Cambridge University Press, 2001.

\bibitem[HH10]{HH}
J\"{u}rgen Herzog and Takayuki Hibi, \emph{Monomial ideals}, Springer, 2010.

\bibitem[HQS15]{HQS}
J{\"u}rgen Herzog, Ayesha~Asloob Qureshi, and Akihiro Shikama, \emph{Alexander
  duality for monomial ideals associated with isotone maps between posets},
  Journal of Algebra and Its Applications, to appear (2015), arXiv preprint
  arXiv:1504.01520.

\bibitem[HRW99]{HRW}
J{\"u}rgen Herzog, Victor Reiner, and Volkmar Welker, \emph{Componentwise
  linear ideals and {G}olod rings}, Michigan Math. J \textbf{46} (1999), no.~2,
  211--223.

\bibitem[HS79]{HS}
J{\"u}rgen Herzog and Manfred Steurich, \emph{Golodideale der {G}estalt
  $\mathfrak{a} \cap \mathfrak{b}$}, J. Algebra \textbf{58} (1979), no.~1,
  31--36.

\bibitem[HT10]{Hoa}
Le~Tuan Hoa and Nguyen~Duc Tam, \emph{On some invariants of a mixed product of
  ideals}, Arch. Math. (Basel) \textbf{94} (2010), no.~4, 327--337.

\bibitem[HTYZN11]{HTYZ}
Hassan Haghighi, Naoki Terai, Siamak Yassemi, and Rahim Zaare-Nahandi,
  \emph{Sequentially {$S_r$} simplicial complexes and sequentially {$S_2$}
  graphs}, Proceedings of the American Mathematical Society \textbf{139}
  (2011), no.~6, 1993--2005.

\bibitem[JK05]{JK}
Sean Jacques and Mordechai Katzman, \emph{The {B}etti numbers of forests},
  arXiv preprint math/0501226 (2005).

\bibitem[J{\"o}l06]{Joe}
Michael J{\"o}llenbeck, \emph{On the multigraded {H}ilbert and
  {P}oincar{\'e}--{B}etti series and the {G}olod property of monomial rings},
  Journal of Pure and Applied Algebra \textbf{207} (2006), no.~2, 261--298.

\bibitem[KH06]{KH}
Masako Kokubo and Takayuki Hibi, \emph{Weakly polymatroidal ideals}, Algebra
  Colloquium \textbf{13} (2006), no.~4, 711--720.

\bibitem[Miy89]{Miya}
Mitsuhiro Miyazaki, \emph{Characterizations of {B}uchsbaum complexes},
  Manuscripta Math. \textbf{63} (1989), no.~2, 245--254.

\bibitem[Pee10]{P}
Irena Peeva, \emph{Graded syzygies}, Springer, 2010.

\bibitem[SFW14]{SFW}
Seyed~Amin Seyed~Fakhari and Volkmar Welker, \emph{The {G}olod property for
  products and high symbolic powers of monomial ideals}, Journal of Algebra
  \textbf{400} (2014), 290--298.

\bibitem[Sta96]{StaCCA}
Richard~P. Stanley, \emph{{C}ombinatorics and {C}ommutative {A}lgebra}, second
  ed., Birkh\"auser, Boston, 1996.

\bibitem[Sta12]{StanleyEnumerative}
\bysame, \emph{Enumerative combinatorics. {V}olume 1}, second ed., Cambridge
  Studies in Advanced Mathematics, vol.~49, Cambridge University Press,
  Cambridge, 2012. \MR{2868112}

\bibitem[Sta13]{StTopicsInAlg}
\bysame, \emph{Algebraic combinatorics}, Undergraduate Texts in Mathematics,
  Springer, New York, 2013.

\bibitem[Ste16]{dSt}
Alessandro~De Stefani, \emph{Products of ideals may not be {G}olod}, Journal of
  Pure and Applied Algebra \textbf{220} (2016), no.~6, 2289--2306.

\bibitem[Ter07]{Te}
Naoki Terai, \emph{Alexander duality in {S}tanley-{R}eisner rings}, Affine
  algebraic geometry (T. Hibi, ed.) (2007), 449--462.

\bibitem[Vil01]{Vi}
Rafael~H. Villarreal, \emph{Monomial algebras}, Marcel Dekker, Inc., 2001.

\bibitem[Wal88]{Walker}
James~W. Walker, \emph{Canonical homeomorphisms of posets}, European Journal of
  Combinatorics \textbf{9} (1988), no.~2, 97--107.

\end{thebibliography}

\end{document}